\documentclass[11pt]{article}

\usepackage{amssymb}
\usepackage{amsmath,amsfonts}
\usepackage{amsthm}
\usepackage{latexsym}
\usepackage{mathrsfs}
\usepackage{color}
\usepackage{algorithm}
\usepackage{algorithmic}
\usepackage{subfigure}
\usepackage{dsfont}
\usepackage[title]{appendix}

\usepackage{hyperref}       
\usepackage{url}            
\usepackage{booktabs}       
\usepackage[a4paper]{geometry}
\usepackage{bm}
\usepackage{graphicx}
\usepackage[numbers]{natbib}
\usepackage{amsopn}

\newcommand{\email}[1]{\protect\href{mailto:#1}{#1}}

\newtheorem{lemma}{Lemma}
\newtheorem{thm}{Theorem}

\newtheorem{prop}{Proposition}

\newcounter{spb}
\DeclareMathOperator{\diag}{Diag}

\DeclareMathOperator{\bldg}{BlkDiag}
\DeclareMathOperator{\sign}{sgn}
\DeclareMathOperator{\dist}{dist}
\DeclareMathOperator{\Proj}{Proj}

\DeclareMathOperator{\dom}{dom}

\newcommand{\be}{\begin{equation}}
\newcommand{\ee}{\end{equation}}

\newcommand{\RNum}[1]{\uppercase\expandafter{\romannumeral #1\relax}}
\newcommand{\rNum}[1]{\lowercase\expandafter{\romannumeral #1\relax}}

\newcommand{\argmin}{\arg\min}

\def\bA{\bm{A}}
\def\tbA{\widetilde{\bm{A}}}
\def\bB{\bm{B}}
\def\bC{\bm{C}}
\def\bE{\bm{E}}
\def\bI{\bm{I}}
\def\bP{\bm{P}}
\def\bQ{\bm{Q}}
\def\tbQ{\widetilde{\bm{Q}}}

\def\bS{\bm{S}}

\def\bU{\bm{U}}
\def\bV{\bm{V}}
\def\bW{\bm{W}}
\def\bX{\bm{X}}
\def\bY{\bm{Y}}
\def\bZ{\bm{Z}}
\def\bSigma{\bm{\Sigma}} 
\def\Labd{\bm{\Lambda}}
\def\b0{\bm{0}}
\def\bq{\bm{q}}
\def\tbq{\tilde{\bq}}
\def\bv{\bm{v}}
\def\bx{\bm{x}}

\def\by{\bm{y}}
\def\bz{\bm{z}}

\def\mB{\mathcal{B}}
\def\mI{\mathcal{I}}
\def\mN{\mathcal{N}}
\def\mQ{\mathcal{Q}}
\def\mO{\mathcal{O}}
\def\mS{\mathcal{S}}
\def\mT{\mathcal{T}}

\def\R{\mathbb{R}}
\def\S{\mathbb{S}}

\begin{document}
\title{Linear Convergence of a Proximal Alternating Minimization Method with Extrapolation for $\ell_1$-Norm \\ Principal Component Analysis\thanks{A preliminary version of this work has appeared in the Proceedings of the 2019 IEEE International Conference on Acoustics, Speech, and Signal Processing (ICASSP 2019)~\cite{wang2019globally}.}} 
\author{ 
Peng Wang\thanks{Department of Systems Engineering and Engineering Management, The Chinese University of
Hong Kong, Shatin N.T., Hong Kong SAR, China. (\email{wangpeng@se.cuhk.edu.hk}).}
\and Huikang Liu\thanks{Imperial College Business School, Imperial College London.(\email{huikang.liu@imperial.ac.uk}).}
\and Anthony Man-Cho So\thanks{Department of Systems Engineering and Engineering Management, The Chinese University of Hong Kong, Shatin N.T., Hong Kong SAR, China. (\email{manchoso@se.cuhk.edu.hk}, \url{http://www.se.cuhk.edu.hk/\~manchoso/}).}}
\date{\today}
\maketitle

\begin{abstract}
A popular robust alternative of the classic principal component analysis (PCA) is the $\ell_1$-norm PCA (L1-PCA), which aims to find a subspace that captures the most variation in a dataset as measured by the $\ell_1$-norm. L1-PCA has shown great promise in alleviating the effect of outliers in data analytic applications. However, it gives rise to a challenging non-smooth non-convex optimization problem, for which existing algorithms are either not scalable or lack strong theoretical guarantees on their convergence behavior. In this paper, we propose a proximal alternating minimization method with extrapolation (PAMe) for solving a two-block reformulation of the L1-PCA problem. We then show that for both the L1-PCA problem and its two-block reformulation, the Kurdyka-\L ojasiewicz exponent at any of the limiting critical points is $1/2$. This allows us to establish the linear convergence of the sequence of iterates generated by PAMe and to determine the criticality of the limit of the sequence with respect to both the L1-PCA problem and its two-block reformulation. To complement our theoretical development, we show via numerical experiments on both synthetic and real-world datasets that PAMe is competitive with a host of existing methods. Our results not only significantly advance the convergence theory of iterative methods for L1-PCA but also demonstrate the potential of our proposed method in applications.
\end{abstract}


\section{Introduction}\label{sec:intro}
Dimension reduction is a powerful paradigm for facilitating information extraction from large datasets. Among the many existing dimension reduction techniques, perhaps the most classic and widely used one is principal component analysis (PCA), which aims to identify a low-dimensional subspace that captures the most variation in the dataset. Concretely, let $\bX=[\bx_1,\ldots,\bx_n] \in \R^{d\times n}$ be the data matrix, where $n$ and $d$ denote the number of samples and dimension of the data points, respectively. Suppose that the sample mean of the dataset $\{\bx_1,\ldots,\bx_n\}$ is zero. Then, a common formulation for finding the desired low-dimensional subspace is given by
\be \label{eq:l2-pca}
\max_{\bQ \in {\rm St}(d,K)}  \|\bX^T \bQ\|_F,
\ee
where $K$ is the dimension of the subspace with $K \le \min\{n,d\}$ and ${\rm St}(d,K) = \{ \bQ \in \mathbb{R}^{d\times K} : \bQ^T\bQ=\bI_K \}$ is the compact Stiefel manifold with $\bI_K$ being the $K\times K$ identity matrix~\cite{jolliffe2016principal}. Despite its non-convexity, Problem~\eqref{eq:l2-pca}, which we shall refer to as L2-PCA, can be solved efficiently by computing the singular value decomposition (SVD) of the data matrix $\bX$. Moreover, the subspace spanned by the columns of an optimal solution to Problem~\eqref{eq:l2-pca} possesses many nice properties~\cite{jolliffe2002principal}. Nevertheless, it has long been known that L2-PCA is sensitive to corruptions in the dataset (generically referred to as \emph{outliers}); see, e.g.,~\cite{devlin1981robust}. This makes L2-PCA ill-suited for many contemporary applications, as the datasets are often collected by automated devices and can be prone to outliers. Over the years, there has been much effort in developing alternatives to L2-PCA that are robust against outliers; see, e.g.,~\cite{lerman2018overview,markopoulos2018outlier} and the references therein. One approach is to replace the $\ell_2$- (Frobenius) norm in~\eqref{eq:l2-pca} with a suitable measure of dispersion in the dataset called \emph{scale function}. Various scale functions and their statistical properties have been studied in the literature; see, e.g.,~\cite{maronna2006robust} and the references therein. In particular, by taking the $\ell_1$-norm as the scale function, we obtain the following robust alternative to L2-PCA, which we shall refer to as L1-PCA:
\be \label{L1-PCA} 
\max_{\bQ \in {\rm St}(d,K)}  \|\bX^T \bQ\|_1.
\ee
Here, $\|\bA\|_1=\sum_{i,j}|A_{ij}|$ denotes the $\ell_1$-norm of the matrix $\bA$. Besides being of interest in its own right, L1-PCA is also related to other data analytic tools, such as independent component analysis~\cite{martin2016link} and linear discriminant analysis~\cite{martin2019lda}. However, unlike L2-PCA, which can essentially be solved in polynomial time, L1-PCA gives rise to a challenging computational problem. Indeed, it is shown in~\cite{mccoy2011two} that Problem~\eqref{L1-PCA} is NP-hard even when $K=1$. This motivates the development of numerically efficient algorithms for solving the L1-PCA problem. 

Many of the earlier algorithms for L1-PCA, such as~\cite{baccini1996l1,choulakian2006L1,croux2007algorithms}, are heuristic in nature. In particular, there is no guarantee that the outputs of these algorithms satisfy any optimality condition of Problem~\eqref{L1-PCA}. Among the first algorithms for L1-PCA that come with theoretical guarantees are those proposed by Kwak~\cite{kwak2008principal} and Nie et al.~\cite{nie2011robust}, which are based on fixed-point (FP) iterations. The former applies to Problem~\eqref{L1-PCA} with $K=1$, while the latter can handle general $K\ge1$. The per-iteration computational costs of these two algorithms are bounded by $\mO(ndK+dK^2)$, which is cheap in the practically relevant case where $K \ll \min\{n,d\}$. Moreover, it is shown that for both algorithms, the iterates generated have a limit point (i.e., subsequential convergence of the iterates) and every limit point satisfies certain first-order optimality condition of the problem. However, the convergence rates of the two algorithms remain unknown. Around the same time, McCoy and Tropp~\cite{mccoy2011two} studied a semidefinite relaxation (SDR) approach (see~\cite{luo2010semidefinite} for an overview) to solving Problem~\eqref{L1-PCA} when $K=1$. It is shown that with high probability, the solution obtained via this approach will have an objective value that is at least $c\cdot\sqrt{2/\pi}$ times the optimal value for any fixed $c\in(0,1)$. However, standard interior-point method-based implementations of the SDR approach have a computational complexity of roughly $\mO(n^{3.5})$, which renders the approach impractical when the dataset is large. Later, Markopoulos et al.~\cite{markopoulos2014optimal} proposed an exact algorithm for solving Problem~\eqref{L1-PCA} that runs in $\mO(n^{dK-K+1})$ time. Although this algorithm is impractical due to its high computational cost, it shows that Problem~\eqref{L1-PCA} is actually polynomial-time solvable when \emph{both} $d$ and $K$ are fixed. Moreover, it can be used to benchmark the solution quality of different L1-PCA algorithms. In a follow-up work, Markopoulos et al.~\cite{markopoulos2017efficient} developed an algorithm based on bit-flipping (BF) iterations for tackling Problem~\eqref{L1-PCA}. On one hand, the computational cost of each BF iteration is $\mO(ndK+nK^3))$, which is inferior to that of the FP iteration developed in~\cite{nie2011robust} when $n \ge d$. On the other hand, the algorithm based on BF iterations is guaranteed to converge in a finite number of steps, while that based on FP iterations is not known to possess such a property. Nevertheless, the number of BF iterations needed can be exponential in $n$ and $K$ in the worst case. Moreover, it is not clear whether the solution obtained from the BF iterations satisfies any optimality condition of Problem~\eqref{L1-PCA}. Recently, Kim and Klabjan~\cite{kim2019simple} revisited Problem~\eqref{L1-PCA} under the setting where $K=1$ and proposed an algorithm similar to those in~\cite{kwak2008principal,nie2011robust} for tackling it. It is shown that the sequence of iterates generated by the algorithm will converge in a finite number of steps. This qualitatively improves upon the subsequential convergence results in~\cite{kwak2008principal,nie2011robust}. Moreover, by pretending that the objective function of~\eqref{L1-PCA} is smooth, it is claimed that the limit of the sequence is a local maximum of the problem. However, a rigorous proof of this claim is still missing.

\subsection{Our Contributions} \label{subsec:contrib}
In view of the above discussion, our goal in this paper is to develop an iterative method for solving Problem~\eqref{L1-PCA} that is numerically efficient and has strong theoretical guarantees on its convergence behavior. To begin, observe that since $|x| = \max\{x,-x\}$ for any $x\in \R$, we can reformulate Problem \eqref{L1-PCA} as
\be \label{L1-PCA-Re} 
\min_{\bP\in \mB(n,K),\,\bQ \in {\rm St}(d,K)} -\langle \bP, \bX^T \bQ \rangle,
\ee
where $\langle \bA,\bB \rangle = {\rm tr}(\bm{A}^T\bB)$ denotes the Euclidean inner product of two matrices $\bA,\bB$ of the same dimensions and $\mB(n,K)=\{\bP \in \R^{n\times K}: P_{ij}\in\{\pm 1\};\ i=1,\ldots,n;\ j=1,\ldots,K\}$ is the set of $n\times K$ $\pm1$ matrices. 
Noting that Problem~\eqref{L1-PCA-Re} has two separate blocks of variables $\bP$ and $\bQ$, we can tackle it using the \emph{proximal alternating minimization} (PAM) method~\cite{auslender1992asymptotic,attouch2010proximal}. To achieve further speed-up, we equip the method with an extrapolation scheme, in which a point extrapolated from previous iterates of the block variable $\bQ$ is used in the update of the block variable $\bP$. It is worth noting that such a scheme differs from those developed for accelerating various proximal block coordinate descent-type methods (see, e.g.,~\cite{pock2016inertial,gao2019gauss,hien2020inertial} and the references therein) and seems to be new. The resulting method, which we call \emph{proximal alternating minimization with extrapolation} (PAMe), admits an efficient implementation, as the update of each block can essentially be given in closed form. In particular, it has a per-iteration computational cost of $\mO(ndK+dK^2)$, which is competitive with the methods based on FP iterations in~\cite{kwak2008principal,nie2011robust,kim2019simple}. Thus, PAMe is well suited to tackle large instances of Problem~\eqref{L1-PCA}. Our numerical experiments on both synthetic and real-world datasets show that PAMe can be significantly faster than PAM and is competitive, in terms of both computational efficiency and solution quality, with a host of existing methods.

To shed light on the numerical performance of and obtain strong theoretical convergence guarantees for our proposed method PAMe, a key step is to characterize the growth behavior of the objective function of~\eqref{L1-PCA-Re} around the \emph{limiting critical points} (see Subsection~\ref{subsec:def} for the definition) of the problem. Towards that end, we first show that the Kurdyka-\L ojasiewicz (K\L) exponent at any limiting critical point of certain orthogonality constrained linear optimization (LO-OC) problem is $1/2$. This result is new and complements that in~\cite{Liu2018} for (homogeneous) quadratic optimization with orthogonality constraint. Moreover, it implies, through a calculus rule established in~\cite{li2017calculus}, that the K\L\ exponent at any limiting critical point of the original L1-PCA formulation~\eqref{L1-PCA} is $1/2$. Then, we relate the limiting critical points of~\eqref{L1-PCA-Re} to those of a particular instance of the LO-OC problem and show that the K\L\ exponent at any of the former is also $1/2$. 
With this characterization, we can utilize the analysis framework in~\cite{attouch2010proximal,attouch2013convergence} to establish the linear convergence of PAMe to a limit $(\bP^*,\bQ^*)$, which is a limiting critical point of Problem~\eqref{L1-PCA-Re}. Moreover, we show that the limit $\bQ^*$ is a \emph{critical point} (see Subsection~\ref{subsec:def} for the definition) of Problem~\eqref{L1-PCA} under certain conditions on the step sizes of PAMe. To the best of our knowledge, our work is the first to determine the K\L\ exponent at the limiting critical points of both~\eqref{L1-PCA} and~\eqref{L1-PCA-Re} and to present a first-order method that provably converges to a limiting critical point of~\eqref{L1-PCA-Re} at a linear rate.

The rest of this paper is organized as follows. In Section~\ref{sec:main}, we introduce our proposed method PAMe and present the main results of this paper.  We then prove the main results in Section~\ref{sec:pf-thm-1} (concerning the K\L\ exponent at the limiting critical points of Problems~\eqref{L1-PCA} and ,\eqref{L1-PCA-Re}) and Section~\ref{sec:pf-thm-2} (concerning the convergence behavior of PAMe). In Section~\ref{sec:num}, we report the numerical performance of PAMe and other existing methods on both synthetic and real-world datasets. We end with some closing remarks in Section~\ref{sec:clud}.

\subsection{Notation and Definitions} \label{subsec:def}
In addition to the notation introduced earlier, we will use the following throughout the paper. Let $\mO^n={\rm St}(n,n)$ denote the set of $n\times n$ orthogonal matrices and $\mathbb{S}^n$ denote the set of $n\times n$ symmetric matrices. For any $a\in\R$, let
\[
\sign(a) \in \left\{
\begin{array}{c@{\,\,\,}l}
\{a/|a|\}, & a \not= 0, \\
\noalign{\smallskip}
\{-1,1\}, & a = 0
\end{array}
\right.
\]
denote (a variant of) the sign function that will be used to express the subdifferential of $x \mapsto -|x|$. Given a matrix $\bA$, let $\sign(\bA)$ denote the matrix obtained by applying $\sign(\cdot)$ to each entry of $\bA$; $\|\bA\|_F$ and $\|\bA\|$ denote the Frobenius norm and spectral norm of $\bA$, respectively; $\lambda_k(\bA)$ denote the $k$-th largest eigenvalue of $\bA$ if $\bA$ is symmetric. 
Given a vector $\bx$, let $\diag(\bx)$ denote the diagonal matrix with $\bx$ as its diagonal. Given square matrices $\bY_1,\dots,\bY_n$, let $\bldg(\bY_1,\dots,\bY_n)$ denote the block diagonal matrix with $\bY_1,\dots,\bY_n$ as its diagonal blocks.

Next, we introduce some concepts in non-smooth analysis that will be needed in our subsequent development. The details can be found in, e.g.,~\cite{RW04}. For a non-empty closed set $\mS \subseteq \R^{p}$, the \emph{indicator function} $\delta_{\mS}: \R^{p} \rightarrow \{0,+\infty\}$ associated with $\mS$ is defined as 
\[
\delta_{\mS}(\bx) = \left\{
\begin{array}{c@{\,\,\,}l}
0, & \bx \in \mS, \\
+\infty, & \mbox{otherwise};
\end{array}
\right.
\]
the \emph{projection} onto $\mS$ is the set-valued mapping $\Proj_\mS:\R^{p}\rightrightarrows\R^{p}$ given by $\Proj_{\mS}(\bx) = \argmin_{\by \in \mS} \|\by-\bx\|_F$; the \emph{distance} between $\mS$ and another non-empty closed set $\mT\subseteq\R^{p}$ is defined as $\dist(\mS,\mT)=\inf_{\bx \in \mS, \, \by \in \mT} \| \bx-\by \|_F$.

Let $f:\R^{p} \rightarrow (-\infty,+\infty]$ be a given function. The \emph{domain} of $f$ is defined as $\dom(f)=\{\bx \in \R^{p}: f(\bx)<+\infty\}$. The function $f$ is said to be \emph{proper} if $\dom(f)\not=\emptyset$. A vector $\bv\in\R^{p}$ is said to be a \emph{Fr\'{e}chet subgradient} of $f$ at $\bx \in \dom(f)$ if
\be \label{eq:frech-subg}
\liminf_{\by\rightarrow\bx, \atop \by\not=\bx} \frac{ f(\by) - f(\bx) - \langle \bv,\by-\bx \rangle }{ \|\by-\bx\|_F } \ge 0.
\ee
The set of vectors $\bv \in \R^p$ satisfying~\eqref{eq:frech-subg} is called the \emph{Fr\'{e}chet subdifferential} of $f$ at $\bx \in \dom(f)$ and denoted by $\widehat{\partial}f(\bx)$. The \emph{limiting subdifferential}, or simply the \emph{subdifferential}, of $f$ at $\bx\in \dom(f)$ is defined as  
\[ 
\partial f(\bx) = \left\{ \bv\in\R^{p}: \exists \bx^k\rightarrow\bx, \, \bv^k \rightarrow \bv \, \mbox{ with } \, f(\bx^k)\rightarrow f(\bx), \, \bv^k\in\widehat{\partial} f(\bx^k) \right\}.
\] 
By convention, if $\bx\not\in\dom(f)$, then $\partial f(\bx) = \emptyset$. The \emph{domain} of $\partial f$ is defined as $\dom(\partial f) = \{\bx \in \R^p: \partial f(\bx) \not= \emptyset\}$. For the indicator function $\delta_{\mS}: \R^{p} \rightarrow \{0,+\infty\}$ associated with the non-empty closed set $\mS \subseteq \R^{p}$, we have
\[ \widehat{\partial}\delta_{\mS}(\bx) = \left\{ \bv\in\R^{p}: \limsup_{\by\rightarrow\bx, \, \by\in\mS \atop \by\not=\bx} \frac{\langle \bv, \by-\bx \rangle}{\|\by-\bx\|_F} \le 0 \right\} \quad\mbox{and}\quad \partial\delta_{\mS}(\bx) = \mN_{\mS}(\bx) \]
for any $\bx \in \mS$, where $\mN_{\mS}(\bx)$ is the \emph{normal cone} to $\mS$ at $\bx$. 

Now, suppose that the function $f:\R^{p} \rightarrow (-\infty,+\infty]$ is proper and lower semicontinuous. A point $\bx\in\R^{p}$ satisfying $\b0 \in \partial f(\bx)$ is called a \emph{limiting critical point} of $f$. By the generalized Fermat rule (see, e.g.,~\cite[Theorem 10.1]{RW04}), a local minimizer of $f$ is a limiting critical point of $f$. The function $f$ is said to have a \emph{K\L\ exponent} of $\theta \in [0,1)$ at the point $\bar{\bx} \in \dom(\partial f)$ if there exist constants $\epsilon, \eta>0$, $\nu \in (0,+\infty]$ such that
\[ \dist(\b0,\partial f(\bx)) \ge \eta( f(\bx) - f(\bar{\bx}) )^{\theta} \]
whenever $\| \bx - \bar{\bx} \|_F \le \epsilon$ and $f(\bar{\bx}) < f(\bx) < f(\bar{\bx}) + \nu$.

Upon writing Problem \eqref{L1-PCA} as
\be \label{L1-PCA-1}
\min_{\bQ \in \R^{d \times K}} \left\{ \ell(\bQ) = -\|\bX^T \bQ\|_1 + \delta_{{\rm St}(d,K)}(\bQ) \right\}
\ee
and invoking the subdifferential calculus rules in~\cite[Chapter 10B]{RW04}, we see that every locally optimal solution $\bQ\in{\rm St}(d,K)$ to Problem~\eqref{L1-PCA} satisfies
\be \label{eq:weak-opt}
\b0 \in - \bX\sign(\bX^T\bQ) + \mN_{{\rm St}(d,K)}(\bQ).
\ee
A point $\bQ \in {\rm St}(d,K)$ satisfying~\eqref{eq:weak-opt} is called a \emph{critical point} of $\ell$. It should be noted that every limiting critical point of $\ell$ is a critical point of $\ell$, but the converse is not known to hold.

\section{Main Results} \label{sec:main}

As mentioned in Subsection~\ref{subsec:contrib}, our strategy for tackling Problem~\eqref{L1-PCA} is to apply a proximal alternating minimization scheme to its two-block reformulation~\eqref{L1-PCA-Re}. Let us now formalize this strategy and introduce our proposed method PAMe.

To begin, observe that Problem~\eqref{L1-PCA-Re} can be written as
\be \label{eq:L1-PCA-2b}
\min_{\bP \in \R^{n \times K},\,\bQ \in \R^{d \times K}} \left\{ h(\bP,\bQ) = -\langle \bP, \bX^T \bQ \rangle + \delta_{\mB(n,K)}(\bP) + \delta_{{\rm St}(d,K)}(\bQ) \right\},
\ee
which is in a form that is amenable to the PAM method developed in~\cite{auslender1992asymptotic} (see also~\cite{attouch2010proximal}). Given the current iterate $(\bP^k,\bQ^k) \in \mB(n,K) \times {\rm St}(d,K)$, the method generates the next iterate $(\bP^{k+1},\bQ^{k+1}) \in \mB(n,K) \times {\rm St}(d,K)$ via
\begin{align}
\bP^{k+1} &\in \argmin_{\bP \in \mB(n,K)} \left\{ -\langle \bP, \bX^T \bQ^k \rangle + \frac{\alpha_k}{2}\|\bP - \bP^k\|_F^2 \right\}, \label{eq:pam-P} \\
\bQ^{k+1} &\in \argmin_{\bQ \in {\rm St}(d,K)} \left\{ -\langle \bP^{k+1}, \bX^T \bQ \rangle +\frac{\beta_k}{2}\|\bQ - \bQ^k\|_F^2 \right\}, \label{update-Q}
\end{align}
where $\alpha_k,\beta_k>0$ are the step sizes. Motivated by the desire to accelerate the PAM iterations, we incorporate an extrapolation step when updating the block variable $\bP$. Specifically, we replace~\eqref{eq:pam-P} by
\be\label{update-P}
\left\{\quad
\begin{split}
\bE^k &= \bQ^k + \gamma_k (\bQ^k-\bQ^{k-1}), \\
\bP^{k+1} &\in \argmin_{\bP \in \mB(n,K)} \left\{ -\langle \bP, \bX^T \bE^k \rangle + \frac{\alpha_k}{2}\|\bP - \bP^k\|_F^2 \right\}, 
\end{split}
\right.
\ee
where $\bE^k \in \R^{d \times K}$ is the point extrapolated from $\bQ^{k-1}$ and $\bQ^k$ and $\gamma_k \in [0,1)$ is the extrapolation parameter. 
Now, note that both $\bP^{k+1}$ in~\eqref{update-P} and $\bQ^{k+1}$ in~\eqref{update-Q} admit closed-form expressions. On one hand, it is easy to verify that
\be\label{closed-form-P}
\bP^{k+1} \in \sign(\bP^k+\bX^T\bE^k/\alpha_k).
\ee
On the other hand, the update~\eqref{update-Q} is an instance of the orthogonal Procrustes problem \cite{schonemann1966generalized}, whose solution is given by
\be\label{closed-form-Q}
\bQ^{k+1} = \bU^{k+1}{\bV^{k+1}}^T.
\ee
Here, $\bU^{k+1} \in {\rm St}(d,K)$ and $\bV^{k+1} \in \mO^K$ are obtained from a thin SVD $\bU^{k+1}\bSigma^{k+1}{\bV^{k+1}}^T$ of $\bQ^k + \bX\bP^{k+1}/\beta_k$. The above development leads to our proposed method PAMe, whose complete description can be found in~\eqref{alg:PAMe}. Since the costs of implementing~\eqref{closed-form-P} and~\eqref{closed-form-Q} are $\mO(ndK)$ and $\mO(dK^2)$, respectively, the per-iteration cost of~\eqref{alg:PAMe} is $\mO(ndK+dK^2)$, which is cheap when $K \ll \min\{n,d\}$.

\begin{algorithm}[!htbp]
	\caption{Proximal Alternating Minimization with Extrapolation (PAMe) for L1-PCA}  
	\begin{algorithmic}[1]  
		\STATE \textbf{Input:} $\bX \in \R^{d \times n}$, $\bP^0 \in \mB(n,K)$, $\bQ^{-1} = \bQ^0 \in {\rm St}(d,K)$ 
		\FOR{$k=0,1,2,\ldots$}
		\STATE  choose step sizes $\alpha_k,\beta_k>0$ and extrapolation parameter $\gamma_k \in [0,1)$;
		\STATE  set $\bE^k \leftarrow \bQ^k+\gamma_k(\bQ^k-\bQ^{k-1})$; \label{line:extrap}
		\STATE  pick $\bP^{k+1} \in \sign(\bP^k+\bX^T\bE^k/\alpha_k)$;
		\STATE  compute a thin SVD $\bQ^k+\bX\bP^{k+1}/\beta_k = \bU^{k+1}\bSigma^{k+1}{\bV^{k+1}}^T$;
		\STATE  set $\bQ^{k+1} \leftarrow \bU^{k+1}{\bV^{k+1}}^T$;
		\STATE  terminate if stopping criteria are met;
		\ENDFOR
	\end{algorithmic}
	\label{alg:PAMe}
\end{algorithm}

Although at first sight the extrapolation scheme introduced above is similar to those used in various inertial proximal block coordinate descent-type methods (see, e.g.,~\cite{pock2016inertial,gao2019gauss,hien2020inertial} and the references therein), there are two crucial differences. First, instead of performing an extrapolation step in each of block updates, PAMe performs such a step in only one of the block updates. Second, in most existing extrapolation schemes, each block update involves an extrapolation point that is obtained from previous iterates of that same block. By contrast, PAMe uses previous iterates of the block variable $\bQ$ to generate an extrapolation point for the update of the block variable $\bP$; see~\eqref{update-P}. As we shall see in Section~\ref{sec:num}, our proposed extrapolation scheme has better numerical performance than existing ones when tackling Problem \eqref{L1-PCA-Re}. It is also interesting to note that the updates in PAMe are similar to those obtained when applying, in a formal manner, proximal difference-of-convex algorithms with extrapolation (see, e.g.,~\cite{wen2018proximal,lu2019enhanced}) to Problem~\eqref{L1-PCA-1}. For instance, the update of the block variable $\bQ$ in the method pDCAe developed in~\cite{wen2018proximal} amounts to projecting $\bE^k + \bX\bm{\xi}^k/\beta_k$ with $\bm{\xi}^k \in \sign(\bX^T\bQ^k)$ onto ${\rm St}(d,K)$, while that in our proposed method PAMe amounts to projecting $\bQ^k + \bX\bm{\xi}^k/\beta_k$ with $\bm{\xi}^k \in \sign(\bP^k + \bX^T\bE^k/\alpha_k)$ onto ${\rm St}(d,K)$ (see~\eqref{update-Q} and~\eqref{closed-form-P}). In fact, for the update in PAMe, we can take $\bm{\xi}^k \in \sign(\bX^T\bE^k)$ when $\alpha_k>0$ is sufficiently small. This further brings out the resemblance between the updates in pDCAe and PAMe. Nevertheless, since the objective function $\ell$ of Problem~\eqref{L1-PCA-1} is not of the difference-of-convex type, existing analyses (such as those in~\cite{wen2018proximal,lu2019enhanced}) do not yield any guarantee on the convergence behavior of proximal difference-of-convex algorithms when applied to Problem~\eqref{L1-PCA-1}. Moreover, we observe that PAMe outperforms pDCAe in our numerical experiments; see Section~\ref{sec:num} for details.

Next, we present the main theoretical contributions of this paper. Our first result states that the objective function $\ell$ of Problem~\eqref{L1-PCA-1} (resp.~$h$ of Problem~\eqref{eq:L1-PCA-2b}) has a K\L\ exponent of $1/2$ at any of its limiting critical points. Combining this with the result in~\cite[Lemma 2.1]{li2017calculus}, we conclude that $\ell$ (resp. $h$) has a K\L\ exponent of $1/2$ at any $\bQ \in \dom(\partial\ell)$ (resp.~$(\bP,\bQ) \in \dom(\partial h)$). This opens the possibility of determining the convergence rates of a host of iterative methods for solving Problems~\eqref{L1-PCA-1} and~\eqref{eq:L1-PCA-2b}; see, e.g.,~\cite{attouch2010proximal,attouch2013convergence}.

\begin{thm}\label{thm:KL-L1-PCA}
Let $\bQ^* \in {\rm St}(d,K)$ be a limiting critical point of Problem~\eqref{L1-PCA-1}. Then, there exist $\epsilon_{\ell}\in(0,1)$ and $\eta_{\ell}>0$ such that for all $\bQ \in {\rm St}(d,K)$ with $\| \bQ-\bQ^* \|_F \le \epsilon_{\ell}$, 
\be \label{exponet-1}
\dist(\b0, \partial \ell(\bQ)) \ge \eta_{\ell} \left| \ell(\bQ) - \ell(\bQ^*) \right|^{1/2}.
\ee
Moreover, let $\bZ^* \in \mB(n,K) \times {\rm St}(d,K)$ be a limiting critical point of Problem~\eqref{eq:L1-PCA-2b}. Then, there exist $\epsilon_h \in (0,1)$ and $\eta_h >0$ such that for all $\bZ \in \mB(n,K) \times {\rm St}(d,K)$ with $\| \bZ-\bZ^* \|_F \le \epsilon_h$, 
\begin{equation}\label{exponet-2}
\dist(\b0, \partial h(\bZ)) \ge \eta_h |h(\bZ) - h(\bZ^*)|^{1/2}.
\end{equation}
\end{thm}
Theorem~\ref{thm:KL-L1-PCA} implies that we can take $\epsilon=\epsilon_\ell$, $\eta=\eta_{\ell}$, and $\nu = +\infty$ (resp.~$\epsilon=\epsilon_h$, $\eta=\eta_h$, and $\nu = +\infty$) in the definition of the K\L\ exponent of $\ell$ (resp.~$h$) at $\bQ^*$ (resp.~$\bZ^*$); see Subsection~\ref{subsec:def}. We remark that the constants $\epsilon_{\ell}$, $\eta_{\ell}$ and $\epsilon_h$, $\eta_h$ can be determined explicitly; see Subsection~\ref{subsec:pf-thm-1}.

With Theorem~\ref{thm:KL-L1-PCA} at our disposal, we can study the convergence behavior of PAMe (Algorithm \eqref{alg:PAMe}). Our second result has two parts. The first part states that with suitable choices of the parameters in PAMe, the iterates $\{(\bP^k,\bQ^k)\}_{k\ge0}$ generated by the method will converge linearly to a limiting critical point $(\bP^*,\bQ^*)$ of Problem~\eqref{eq:L1-PCA-2b}. Now, since our original problem of interest is Problem~\eqref{L1-PCA-1}, a natural question would be whether $\bQ^*$ is one of its (limiting) critical points. Unfortunately, we do not yet know the answer to this question. The second part of our result, which provides a partial answer, gives a sufficient condition for $\bQ^*$ to be a critical point of Problem~\eqref{L1-PCA-1} (i.e., $\bQ^*$ satisfies~\eqref{eq:weak-opt}).

\begin{thm}\label{thm:liner-conv}
Let $\{(\bP^k,\bQ^k)\}_{k\ge0}$ be the sequence of iterates generated by Algorithm~\ref{alg:PAMe}, where the step sizes $\{\alpha_k\}_{k\ge0}$, $\{\beta_k\}_{k\ge0}$ and extrapolation parameters $\{\gamma_k\}_{k\ge0}$ satisfy (i) $\alpha_* \le \alpha_k \le \alpha^*$ for some $\alpha_*,\alpha^* \in (0,+\infty)$, (ii) $3\beta_*/2 \le \beta_k \le \beta^*$ for some $\beta_*,\beta^* \in (0,+\infty)$, and (iii) $0 \le \gamma_k < \gamma^* = \min\{ 1, \alpha_*\beta_*/2\|\bX\|^2 \}$. Then, the sequence $\{(\bP^k,\bQ^k)\}_{k\ge0}$ converges at least linearly to a limiting critical point $(\bP^*,\bQ^*)$ of Problem~\eqref{eq:L1-PCA-2b}. 
Moreover, if $\alpha_k = \alpha_*$ for $k\ge0$, then $\bQ^*$ is a solution to the following generalized equation:
\be \label{eq:weak-fpt-incl}
\b0 \in -\bX\sign(\bP^* + \bX^T\bQ/\alpha_*) + \mN_{{\rm St}(d,K)}(\bQ).
\ee
In particular, if $\alpha_*$ satisfies
\be \label{eq:alpha-range}
0 <  \alpha_* < \min\{ |(\bX^T\bQ^*)_{ij}| : (\bX^T\bQ^*)_{ij} \not= 0; \, i=1,\ldots,n; \, j=1,\ldots,K\}, 
\ee
then $\bQ^*$ is a critical point of Problem~\eqref{L1-PCA-1}. Conversely, every critical point $\bar{\bQ}$ of Problem~\eqref{L1-PCA-1} satisfying $\b0 \in -\bX\bP^* + \mN_{{\rm St}(d,K)}(\bar{\bQ})$ and $\bP^* \in \sign(\bX^T\bar{\bQ})$ is a solution to the generalized equation~\eqref{eq:weak-fpt-incl}, regardless of whether~\eqref{eq:alpha-range} holds.
\end{thm}
It is worth noting that the sufficient condition~\eqref{eq:alpha-range} is \emph{efficiently verifiable}; i.e., after obtaining the limit point $\bQ^*$, one can efficiently verify whether~\eqref{eq:alpha-range} holds. Moreover, condition~\eqref{eq:alpha-range} suggests that PAMe is more likely to return a critical point of Problem~\eqref{L1-PCA-1} if we choose a smaller step size $\alpha_*$. In fact, we observe from our numerical experiments that a small $\alpha_*$ often leads to favorable performance of PAMe on the L1-PCA problem; see Section~\ref{sec:num} for details.

\section{Characterizing the K\L\ exponent for Problems \eqref{L1-PCA-1} and \eqref{eq:L1-PCA-2b}} \label{sec:pf-thm-1}

Our goal in this section is to prove Theorem~\ref{thm:KL-L1-PCA}. This is achieved in three steps. First, we invoke a calculus rule established in~\cite{li2017calculus} to show that the task of estimating the K\L\ exponent at a limiting critical point of Problem~\eqref{L1-PCA-1} reduces to that of estimating the K\L\ exponent at a limiting critical point of an LO-OC problem. Then, we establish a local error bound for the LO-OC problem and use it to characterize the K\L\ exponent for that problem. Lastly, we utilize the result obtained for the LO-OC problem and the structures of Problems \eqref{L1-PCA-1} and \eqref{eq:L1-PCA-2b} to complete the proof.

\subsection{Relation with Linear Optimization over the Stiefel Manifold} \label{subsec:kl-calc}
Let $\bP_1,\ldots,\bP_{2^{nK}}$ be an enumeration of the elements in $\mB(n,K)$. By definition of the $\ell_1$-norm, we can express the objective function $\ell$ of Problem~\eqref{L1-PCA-1} as the pointwise minimum of finitely many proper, lower semicontinuous functions:
\[ 
\ell(\bQ) = \min_{i \in \{1,\ldots,2^{nK}\}} \big\{ \underbrace{\langle \bX\bP_i, \bQ \rangle + \delta_{{\rm St}(d,K)}(\bQ)}_{\ell_i(\bQ)} \big\}. \] 
Since $\dom(\partial\ell) \subseteq \dom(\ell) = {\rm St}(d,K)$ by definition and $\ell(\bQ) = -\| \bX^T\bQ \|_1$ for any $\bQ \in {\rm St}(d,K)$, it is immediate that $\ell$ is continuous on $\dom(\partial\ell)$. Moreover, using the result in~\cite[Exercise 8.8]{RW04}, we have $\partial\ell_i(\bQ) = \bX\bP_i + \mN_{{\rm St}(d,K)}(\bQ)$ for any $\bQ \in {\rm St}(d,K)$, which implies that $\dom(\ell_i) = \dom(\partial\ell_i) = {\rm St}(d,K)$ for $i=1,\ldots,2^{nK}$. Thus, by~\cite[Corollary 3.1]{li2017calculus}, in order to determine the K\L\ exponent of the function $\ell$ at a point $\bar{\bx} \in \dom(\partial \ell)$, it suffices to determine the K\L\ exponents of the functions $\ell_1,\ldots,\ell_{2^{nK}}$ at the point $\bar{\bx}$. 
Noting that $\ell_i$ ($i=1,\ldots,2^{nK}$) is the sum of a linear function and the indicator function associated with ${\rm St}(d,K)$, a natural approach is to study the following general LO-OC problem, where $\bA \in \R^{d \times K}$ is any given matrix:
\begin{align}\label{LP-OC} \tag{{\sf LO-OC}}
\min_{\bQ \in \R^{d\times K}} \left\{ g(\bQ) = \langle \bA,\bQ \rangle + \delta_{{\rm St}(d,K)}(\bQ) \right\}.
\end{align} 
By~\cite[Lemma 2.1]{li2017calculus}, for any $\theta \in [0,1)$, the function $g$ has a K\L\ exponent of $\theta$ at any of its non-limiting critical point. Thus, we shall focus on determining the K\L\ exponents of $g$ at its limiting critical points. 

\subsection{Estimating the K\L\ Exponent for Problem~\eqref{LP-OC}}
Let
\[ \mQ = \{ \bQ \in {\rm St}(d,K): \b0 \in \partial g(\bQ) \} \]
denote the set of limiting critical points of Problem~\eqref{LP-OC}. Based on the development in the previous subsection, our next step is to prove the following result, which can be of independent interest.
\begin{thm} \label{thm:KL-LO-OC}
There exist $\epsilon_g \in (0,1)$, $\eta_g>0$ such that for all $\bQ \in {\rm St}(d,K)$ and $\bQ^* \in \mQ$ with $\|\bQ - \bQ^*\|_F \le \epsilon_g$,
\[
\dist\left( \b0,\partial g(\bQ) \right) \ge \eta_g \left| g(\bQ) - g(\bQ^*) \right|^{1/2}.
\]
\end{thm}
Theorem~\ref{thm:KL-LO-OC} implies that the K\L\ exponent at any limiting critical point of Problem~\eqref{LP-OC} is $1/2$ with $\epsilon=\epsilon_g$, $\eta=\eta_g$, and $\nu=+\infty$. It is worth noting that the constants $\epsilon,\eta,\nu$ are uniform over all limiting critical points in $\mQ$. 

The proof of Theorem~\ref{thm:KL-LO-OC} can be divided into three parts. As it is quite long and technical, readers who are interested in how Theorem~\ref{thm:KL-LO-OC} is used to complete the proof of Theorem~\ref{thm:KL-L1-PCA} can skip ahead to Subsection~\ref{subsec:pf-thm-1}.

\subsubsection{Structure of the Limiting Critical Point Set}
We begin with the following result, which provides, among other things, a characterization of $\mQ$.
\begin{prop}\label{prop:subd-G}
Consider the map $R:{\rm St}(d,K) \rightarrow \R^{d \times K}$ given by
\be \label{R(Q)}
R(\bQ) = \bA - \bQ\bA^T\bQ.
\ee
We have
\be\label{rst-1:prop-subd-G}
\dist(\b0,\partial g(\bQ)) = \left\| \left(\bI_d-\frac{1}{2}\bQ\bQ^T\right)R(\bQ)\right\|_F 
\ee
and
\[
\frac{1}{2}\|R(\bQ)\|_F \le \dist(\b0,\partial g(\bQ)) \le \|R(\bQ)\|_F.
\]
In particular, we have $\bQ \in \mQ$ if and only if
\be\label{rst-2:prop-subd-G}
\bQ \in {\rm St}(d,K) \quad\mbox{and}\quad R(\bQ) = \b0.
\ee
\end{prop}
\begin{proof}
By the result in \cite[Exercise 8.8]{RW04}, we have $\partial g(\bQ) = \bA + \mN_{{\rm St}(d,K)}(\bQ)$ for any $\bQ \in {\rm St}(d,K)$. Since the differential of the map $\R^{d \times K} \ni \bX \mapsto \bX^T\bX - \bI_K \in \mathbb{S}^K$ has full rank (see the discussion in~\cite[Chapter 3.3.2]{absil2009optimization}), we can invoke the result in~\cite[Example 6.8]{RW04} to obtain $\mN_{{\rm St}(d,K)}(\bQ) = \mT_{{\rm St}(d,K)}(\bQ)^\perp$, where $\mT_{{\rm St}(d,K)}(\bQ)$
is the tangent space to ${\rm St}(d,K)$ at $\bQ \in {\rm St}(d,K)$. In particular, using the decomposition $\bA = \Proj_{\mT_{{\rm St}(d,K)}(\bQ)}(\bA) + \Proj_{\mN_{{\rm St}(d,K)}(\bQ)}(\bA)$ (see~\cite[Chapter 3.6.1]{absil2009optimization}) and the formula $\Proj_{\mT_{{\rm St}(d,K)}(\bQ)}(\bA)=\left(\bI_d-\bQ\bQ^T/2\right)R(\bQ)$ (see~\cite[Example 3.6.2]{absil2009optimization}), we have
\begin{align*}
\dist(\b0,\partial g(\bQ)) & = \inf_{\bS \in \mN_{{\rm St}(d,K)}(\bQ)} \|\bA + \bS\|_F  = \|\Proj_{\mT_{{\rm St}(d,K)}(\bQ)}(\bA)\|_F \\
& = \left\| \left(\bI_d-\frac{1}{2}\bQ\bQ^T\right)R(\bQ) \right\|_F.
\end{align*}
Now, observe that $\bI_d-\bQ\bQ^T/2$ is invertible and the eigenvalues of $\bQ\bQ^T$ are $0$ or $1$. It follows that
\begin{align*}
\|R(\bQ)\|_F &= \left\| \left(\bI_d-\frac{1}{2}\bQ\bQ^T\right)^{-1} \left(\bI_d-\frac{1}{2}\bQ\bQ^T\right)R(\bQ) \right\|_F  \\
& \le\ 2\left\| \left(\bI_d-\frac{1}{2}\bQ\bQ^T\right)R(\bQ) \right\|_F
\end{align*}
and
\begin{align*}
\left\| \left(\bI_d-\frac{1}{2}\bQ\bQ^T\right)R(\bQ) \right\|_F \le \|R(\bQ)\|_F.
\end{align*}
Putting the above pieces together, we obtain
\begin{align*}
\frac{1}{2}\|R(\bQ)\|_F \le \dist(\b0,\partial g(\bQ)) \le \|R(\bQ)\|_F,
\end{align*}
as desired.
\end{proof}

Now, suppose that the rank of $\bA \in \R^{d\times K}$ is $r$, where $r \in \{1,\ldots, K\}$ so that $\bA \not= \b0$ (if $\bA=\b0$, then Theorem \ref{thm:KL-LO-OC} holds trivially). Let
\be \label{eq:A-svd}
\bA=\bU_{\bA}\bSigma_{\bA}\bV_{\bA}^T=
\begin{bmatrix}
\bU_{\bA,1} & \bU_{\bA,2}
\end{bmatrix}
\begin{bmatrix}
\widetilde{\bSigma}_{\bA} & \b0 \\ 
\b0 & \b0 
\end{bmatrix} 
\begin{bmatrix}
\bV_{\bA,1}^T \\
\bV_{\bA,2}^T 
\end{bmatrix}
\ee
be an SVD of $\bA$, where $\widetilde{\bSigma}_{\bA} = \diag(\sigma_1,\ldots,\sigma_r)$ with $\sigma_1 \ge \cdots \ge \sigma_r > 0 $ being the positive singular values of $\bA$; $\bU_{\bA} \in \mO^{d}$ with $\bU_{\bA,1} \in \R^{d\times r}$, $\bU_{\bA,2} \in \R^{d\times (d-r)}$; $\bV_{\bA} \in \mO^K$ with $\bV_{\bA,1} \in \R^{K\times r}$, $\bV_{\bA,2} \in \R^{K \times (K-r)}$. Then, for any $\bQ \in {\rm St}(d,K)$, we have $\langle \bA, \bQ \rangle = \langle \bSigma_{\bA}, \bU_{\bA}^T \bQ \bV_{\bA} \rangle$ with $\bU_{\bA}^T \bQ \bV_{\bA} \in {\rm St}(d,K)$. Moreover, Proposition \ref{prop:subd-G} implies that $\bar{\bQ} \in {\rm St}(d,K)$ is a limiting critical point of the function $\bQ \mapsto \langle \bA,\bQ \rangle + \delta_{{\rm St}(d,K)}(\bQ)$ if and only if $\bU_{\bA}^T \bar{\bQ} \bV_{\bA} \in {\rm St}(d,K)$ is a limiting critical point of the function $\bQ \mapsto \langle \bSigma_{\bA}, \bQ \rangle + \delta_{{\rm St}(d,K)}(\bQ)$. Thus, we can assume without loss of generality that 
\be\label{block-A}
\bA = 
\begin{bmatrix}
\widetilde{\bA} & \b0\\ 
\b0 & \b0
\end{bmatrix},
\ee
where $\widetilde{\bA}=\diag(a_1,\dots,a_r)$ with $a_1 \ge \dots \ge a_r >0$. Suppose that $\bA$ has $p \ge1$ distinct positive singular values. In other words, there exist indices $s_0,s_1,\ldots,s_p$ such that $0=s_0 < s_1 < \cdots < s_p = r$ and 
\be\label{rela-a}
a_{s_0+1}=\cdots = a_{s_1} > a_{s_1+1} = \cdots = a_{s_2} > \cdots >  a_{s_{p-1}+1} = \cdots = a_{s_p}.
\ee
Let $h_i=s_i-s_{i-1}$ be the multiplicity of the $i$-th largest positive singular value, where $i=1,\dots,p$. Then, we clearly have $\sum_{i=1}^p h_i=r$ and
\be \label{block-tA}
\widetilde{\bA} = \bldg(a_{s_1}\bI_{h_1},\ldots,a_{s_p}\bI_{h_p}).
\ee
Based on the block structures of $\bA$ in \eqref{block-A} and $\widetilde{\bA}$ in \eqref{block-tA}, let us partition $\bQ \in \R^{d\times K}$ as
\be \label{block-Q}
\bQ = \begin{bmatrix}
\bQ_1 & \bQ_2 \\
\bQ_3 & \bQ_4 
\end{bmatrix} \quad\mbox{and}\quad
\bQ_1 = \begin{bmatrix}
\bQ_{h_1 h_1} & \dots & \bQ_{h_1h_p}\\
\vdots & \ddots & \vdots \\
\bQ_{h_p h_1} & \dots & \bQ_{h_p h_p}
\end{bmatrix},
\ee
where $\bQ_1\in \R^{r\times r}$, $\bQ_2\in \R^{r\times (K-r)}$, $\bQ_3\in \R^{(d-r)\times r}$, $\bQ_4\in \R^{(d-r)\times (K-r)}$, and $\bQ_{h_ih_j}\in \R^{h_i\times h_j}$ for $i,j\in\{1,\ldots,p\}$. With the above partition, we can further elucidate the structure of a limiting critical point of Problem \eqref{LP-OC}. Specifically, we establish the following result: 

\begin{prop}\label{prop:crit-LP-OC}
Suppose that $\bA \in \R^{d \times K}$ has the form given in~\eqref{block-A}. Then, we have $\bQ \in \mQ$ if and only if 
\begin{align}\label{result:prop-crit-LP-OC}
\bQ = \begin{bmatrix}
\bU \diag(\bq)\bU^T & \b0 \\
\b0 & \bV
\end{bmatrix},
\end{align}
where $\bU = \bldg(\bU_{1},\dots,\bU_{p})$ with $\bU_{i} \in \mO^{h_i}$ for $i=1,\dots,p$, $\bq \in \{\pm 1\}^{r}$, and $\bV \in {\rm St}(d-r,K-r)$.
\end{prop}
\begin{proof}
If $\bQ$ is of the form given in~\eqref{result:prop-crit-LP-OC}, then using the block structure of $\bA$ in~\eqref{block-A}, it is straightforward to verify that $\bQ \in {\rm St}(d,K)$ and $\bA-\bQ\bA^T\bQ=\b0$. By Proposition \ref{prop:subd-G}, we conclude that $\bQ \in \mQ$.

Conversely, suppose that $\bQ \in \mQ$. By Proposition~\ref{prop:subd-G}, we have $\bA-\bQ\bA^T\bQ=\b0$. Since $\bQ \in {\rm St}(d,K)$, this implies that 
\be\label{eq-2:lem-crit-LP-OC}
\bQ^T\bA-\bA^T\bQ = \b0,
\ee
which, together with $\bA-\bQ\bA^T\bQ=\b0$, yields
\be\label{eq-3:lem-crit-LP-OC}
\bA-\bQ\bQ^T\bA = \b0. 
\ee
Using the block structures of $\bA$ in \eqref{block-A} and $\bQ$ in \eqref{block-Q}, we have
\begin{align*}
\bQ^T\bA = \begin{bmatrix}
\bQ_1^T\widetilde{\bA} & \b0 \\
\bQ_2^T\widetilde{\bA} & \b0 
\end{bmatrix}.
\end{align*}
It then follows from~\eqref{eq-2:lem-crit-LP-OC} that $\bQ_1^T\widetilde{\bA} = \widetilde{\bA}^T\bQ_1$ and $\bQ_2^T\widetilde{\bA}=\b0$. Since $\widetilde{\bA}$ has full rank, the latter implies that $\bQ_2=\b0$, which in turn implies that $\bQ_4^T\bQ_4=\bI_{K-r}$ because we have $\bQ \in {\rm St}(d,K)$. Using $\bQ_2=\b0$ and~\eqref{eq-3:lem-crit-LP-OC}, we obtain
\[
\bA-\bQ\bQ^T\bA = \begin{bmatrix}
\widetilde{\bA} - \bQ_1\bQ_1^T\widetilde{\bA} & \b0 \\
-\bQ_3\bQ_1^T\widetilde{\bA} & \b0 
\end{bmatrix} = \b0;
\]
i.e., $(\bI_r-\bQ_1\bQ_1^T)\tbA=\b0$ and $\bQ_3\bQ_1^T\tbA=\b0$. These, together with the fact that $\tbA$ has full rank, imply that $\bQ_1\in \mO^r$ and $\bQ_3=\b0$. 

Now, using the block structures of $\tbA$ in \eqref{block-tA} and $\bQ_1$ in \eqref{block-Q}, we get
\be \label{eq:sym-diff}
\bQ_1^T\tbA = \begin{bmatrix}
a_{s_1} \bQ_{h_1 h_1}^T & \dots & a_{s_p} \bQ_{h_p h_1}^T \\
\vdots & \ddots & \vdots \\
a_{s_1} \bQ_{h_1 h_p}^T & \dots & a_{s_p} \bQ_{h_p h_p}^T
\end{bmatrix}, \quad
\tbA^T\bQ_1 = \begin{bmatrix}
a_{s_1} \bQ_{h_1 h_1} & \dots & a_{s_1} \bQ_{h_1 h_p} \\
\vdots & \ddots & \vdots \\
a_{s_p} \bQ_{h_p h_1} & \dots & a_{s_p} \bQ_{h_p h_p}
\end{bmatrix}.
\ee
Since $\bQ_1^T\tbA=\tbA^T\bQ_1$, we have
\be\label{eq-4:lem-crit-LP-OC}
\bQ_{h_ih_j}^T = \frac{a_{s_j}}{a_{s_i}}\bQ_{h_jh_i} \quad \mbox{for } i,j \in \{1,\ldots,p\},
\ee
which implies that
\be \label{eq-11:lem-crit-LP-OC}
\sum_{i=1}^p \bQ_{h_ih_j}^T\bQ_{h_ih_j} = \sum_{i=1}^p \frac{a_{s_j}^2}{a_{s_i}^2}\bQ_{h_jh_i}\bQ_{h_jh_i}^T \quad \mbox{for } j \in \{1,\ldots,p\}.
\ee
Moreover, the fact that $\bQ_1 \in \mO^r$ implies
\be \label{eq-5:lem-crit-LP-OC}
\sum_{i=1}^p \bQ_{h_ih_j}^T\bQ_{h_ih_j} = \bI_{h_j},  \quad \sum_{i=1}^p \bQ_{h_jh_i}\bQ_{h_jh_i}^T = \bI_{h_j} \quad \mbox{for } j \in \{1,\ldots,p\}.
\ee
It then follows from~\eqref{eq-11:lem-crit-LP-OC} and \eqref{eq-5:lem-crit-LP-OC} that
\be \label{eq-6:lem-crit-LP-OC}
\sum_{i\neq j}^p \left( 1 - \frac{a_{s_j}^2}{a_{s_i}^2} \right)\left\|\bQ_{h_jh_i}\right\|_F^2 = 0 \quad \text{for } j \in \{1,\dots,p\}.
\ee
By rewriting~\eqref{eq-4:lem-crit-LP-OC} as $\bQ_{h_jh_i}^T = \tfrac{a_{s_i}}{a_{s_j}}\bQ_{h_ih_j}$ for $i,j \in \{1,\ldots,p\}$ and repeating the above argument, we get
\be \label{eq-7:lem-crit-LP-OC}
\sum_{i\neq j}^p \left( 1 - \frac{a_{s_i}^2}{a_{s_j}^2} \right)\left\|\bQ_{h_ih_j}\right\|_F^2 = 0 \quad \text{for } j \in \{1,\dots,p\}.
\ee
Since $a_{s_1} > \dots > a_{s_p} > 0$ by \eqref{rela-a}, the identities in~\eqref{eq-6:lem-crit-LP-OC} and \eqref{eq-7:lem-crit-LP-OC} imply that
\[ 
\bQ_{h_ih_j}=\b0 \quad \text{for } i,j \in \{1,\ldots,p\}; \, i\not= j.
\] 
This, together with \eqref{eq-4:lem-crit-LP-OC} and \eqref{eq-5:lem-crit-LP-OC}, yields
\be \label{eq-9:lem-crit-LP-OC}
\bQ_{h_ih_i}=\bQ_{h_ih_i}^T \quad\mbox{and}\quad \bQ_{h_ih_i}^T\bQ_{h_ih_i}=\bI_{h_i} \quad \mbox{for } i \in \{1,\ldots,p\}.
\ee
Let $\bQ_{h_i h_i}=\bU_{i}\bm{\Lambda}_{i}\bU_{i}^T$ ($i=1,\ldots,p$) be an eigen-decomposition of $\bQ_{h_i h_i}$, where $\bU_{i}\in \mO^{h_i}$ and $\bm{\Lambda}_{i} = \diag(\lambda_{s_{i-1}+1},\dots,\lambda_{s_i})$. Then, we have $\bQ_{h_i h_i}^T\bQ_{h_ih_i}=\bU_{i}\bm{\Lambda}^2_{i}\bU_{i}^T=\bI_{h_i}$ from~\eqref{eq-9:lem-crit-LP-OC}, which implies that $\bm{\Lambda}^2_{i}=\bI_{h_i}$. It follows that $\lambda_{s_{i-1}+1},\ldots,\lambda_{s_i} \in \{\pm1\}$. 

Putting all the pieces together, we see that $\bQ_1$ takes the form
\[ 
\bQ_1=\bldg(\bU_{1},\ldots,\bU_{p})\cdot \diag(\bq) \cdot\bldg(\bU_{1}^T,\ldots,\bU_{p}^T)
\] 
with $\bU_{i} \in \mO^{h_i}$ for $i=1,\ldots,p$ and $\bq \in\{\pm 1\}^r$, $\bQ_2=\b0$, $\bQ_3=\b0$, and $\bQ_4 \in {\rm St}(d-r,K-r)$. This completes the proof.
\end{proof}

Proposition \ref{prop:crit-LP-OC} suggests that when $\bA \in \R^{d \times K}$ has the form given in~\eqref{block-A}, the set $\mQ$ of limiting critical points of Problem~\eqref{LP-OC} can be expressed as
\[ \mQ = \bigcup_{\bq \in \{\pm 1\}^r} \mQ_{\bq}, \]
where
\be \label{set:Q-p}
\mQ_{\bq} = \left\{
\begin{bmatrix}
\bU \diag(\bq)\bU^T & \b0 \\
\b0 & \bV
\end{bmatrix} :
\begin{array}{l}
\bU = \bldg(\bU_{1},\dots,\bU_{p}), \\ 
\noalign{\smallskip}
\bU_i\in \mO^{h_i} \mbox{ for } i \in \{1,\ldots,p\}, \\
\noalign{\smallskip}
\bV \in {\rm St}(d-r,K-r)
\end{array}
\right\}.
\ee
The following result shows that the collection $\{ \mQ_{\bq} \}_{\bq \in \{\pm1\}^r}$ essentially forms a well-separated partition of the set $\mQ$.
\begin{prop}\label{prop:dist-Q}
Suppose that $\bA \in \R^{d \times K}$ has the form given in~\eqref{block-A}. Let $\bq,\bq^\prime \in \{\pm 1\}^r$ be arbitrary. Then, we either have $\mQ_{\bq}=\mQ_{\bq^\prime}$ or $\mQ_{\bq}\cap \mQ_{\bq^\prime} = \emptyset$. Moreover, if the latter holds, then $\dist(\mQ_{\bq},\mQ_{\bq^\prime}) \ge 2$.
\end{prop}
\begin{proof}
Let $\bq = (\bq_1,\ldots,\bq_p)$ and $\bq^\prime = (\bq_1^\prime,\ldots,\bq_p^\prime)$, where $\bq_i,\bq_i^\prime \in \{\pm1\}^{h_i}$ for $i=1,\ldots,p$. Suppose that $\bQ =\begin{bmatrix}
\bQ_1 & \b0 \\
\b0 & \bV
\end{bmatrix} \in \mQ_{\bq}$. By definition of $\mQ_{\bq}$ in \eqref{set:Q-p}, for $i=1,\ldots,p$, the eigenvalues of the $i$-th diagonal block of $\bQ_1$ are given by the entries of $\bq_i$. Thus, if $\bQ \in \mQ_{\bq} \cap \mQ_{\bq^\prime}$, then both $\bq_i$ and $\bq_i^\prime$ are vectors of eigenvalues of the $i$-th diagonal block of $\bQ_1$, which implies that $\bq_i$ and $\bq_i^\prime$ are equal up to a permutation for $i=1,\ldots,p$. It follows that whenever $\mQ_{\bq}\cap \mQ_{\bq^\prime} \not= \emptyset$, we have $\mQ_{\bq}=\mQ_{\bq^\prime}$. 

Now, suppose that $\mQ_{\bq}\cap \mQ_{\bq^\prime}=\emptyset$. Let 
\[
\bQ = \begin{bmatrix}
\bU\diag(\bq)\bU^T & \b0 \\
\b0 & \bV
\end{bmatrix} \in \mQ_{\bq}, \quad 
\bQ^\prime = \begin{bmatrix}
\bU^\prime \diag(\bq^\prime) {\bU^\prime}^T & \b0 \\
\b0 & \bV^\prime
\end{bmatrix} \in \mQ_{\bq^\prime}
\]
be arbitrary, where $\bU = \bldg(\bU_1,\dots,\bU_p)$, $\bU^\prime = \bldg(\bU_1^\prime,\dots,\bU_p^\prime)$ with $\bU_i, \bU_i^\prime \in \mO^{h_i}$ for $i=1,\ldots,p$ and $\bV,\bV^\prime \in {\rm St}(d-r,K-r)$. Then, we have
\begin{align} \label{eq-1:prop-dist-Q}
\|\bQ-\bQ^\prime\|_F^2 & = \sum_{i=1}^p \left\|\bU_i\diag(\bq_i)\bU_i^T - \bU_i^\prime \diag(\bq_i^\prime) {\bU_i^\prime}^T \right\|_F^2 + \|\bV-\bV^\prime\|_F^2 \notag \\
& \ge \sum_{i=1}^p \min_{\bU_i\in \mO^{h_i}}\left\|\bU_i\diag(\bq_i) \bU_i^T -\diag(\bq_i^\prime) \right\|_F^2 + \min_{\bV \in {\rm St}(d-r,K-r)} \|\bV-\bV^\prime\|_F^2  \notag \\ 
& = \sum_{i=1}^p \min_{\bU_i\in \mO^{h_i}}\left\|\bU_i\diag(\bq_i) \bU_i^T -\diag(\bq_i^\prime) \right\|_F^2,
\end{align}
where the last equality follows from the fact that $\bV^\prime \in {\rm St}(d-r,K-r)$. For $i=1,\ldots,p$, let $t_i$ and $t_i^\prime$ denote the number of 1's in $\bq_i$ and $\bq_i^\prime$, respectively. If there exists a $j \in \{1,\dots,p\}$ such that $t_j\neq t_j^\prime$, then we can find a $k \in \{1,\dots,h_j\}$ such that for any $\bU_j \in \mO^{h_j}$,
\begin{align}
\left\|\bU_j\diag(\bq_j) \bU_j^T -\diag(\bq_j^\prime) \right\|_F^2 & \ge \left\|\bU_j\diag(\bq_j) \bU_j^T -\diag(\bq_j^\prime) \right\|^2 \nonumber \\ 
& \ge \left| \lambda_k \left( \bU_j\diag(\bq_j) \bU_j^T \right) - \lambda_k(\diag(\bq_j^\prime)) \right|^2 \nonumber \\
& = 4, \label{eq:distge4}
\end{align}
where the second inequality follows from classic perturbation results for eigenvalues of symmetric matrices (see, e.g.,~\cite[Corollary 4.10]{SS90}) and the last equality is due to the fact that $\bq_j,\bq_j^\prime \in \{\pm1\}^{h_j}$ and $t_j \not= t_j^\prime$. Since $\bQ \in \mQ_{\bq}$, $\bQ^\prime \in \mQ_{\bq^\prime}$ are arbitrary, we conclude from~\eqref{eq-1:prop-dist-Q} and \eqref{eq:distge4} that
\[
\dist(\mQ_{\bq},\mQ_{\bq^\prime}) \ge 2.
\]
Otherwise, we have $t_i=t_i^\prime$ for $i=1,\ldots,p$, which implies that $\bq_i$ and $\bq_i^\prime$ are equal up to a permutation for $i=1,\ldots,p$. In this case, we have $\mQ_{\bq}=\mQ_{\bq^\prime}$, which contradicts our assumption that $\mQ_{\bq} \cap \mQ_{\bq^\prime} = \emptyset$. This completes the proof.
\end{proof}

\subsubsection{Local Error Bound}
Equipped with the results in the previous section, our next task is to establish the following local error bound for Problem~\eqref{LP-OC}, which provides an estimate of the distance between any point from a certain subset of ${\rm St}(d,K)$ to the set of limiting critical points of Problem~\eqref{LP-OC} using the map $R$ introduced in~\eqref{R(Q)}. As we shall see, such an error bound plays a crucial role in determining the K\L\ exponent at the limiting critical points of Problem~\eqref{LP-OC}.
\begin{thm}\label{lem:EB}
Let $\bA \in \R^{d \times K}$ be an arbitrary rank-$r$ matrix whose SVD is given by~\eqref{eq:A-svd} and whose positive singular values are given by~\eqref{rela-a}. Then, for any $\bq\in\{\pm 1\}^r$, we have
\be\label{EB:LP-OC}
\dist(\bQ, \bar{\mQ}_{\bq}) \le \kappa\|R(\bQ)\|_F \quad \mbox{for all } \bQ\in{\rm St}(d,K) \mbox{ with } \dist(\bQ, \bar{\mQ}_{\bq}) < 1,
\ee
where $\bar{\mQ}_{\bq} = \bU_{\bA} \mQ_{\bq} \bV_{\bA}^T = \{ \bU_{\bA} \bm{W} \bV_{\bA}^T : \bm{W} \in \mQ_{\bq} \}$ and
\begin{align*}
\kappa &= \frac{1}{a_r} \left( 13 + 6(6p-5)  \left( \min_{i,j\in\{1,\ldots,p\} \atop i\not=j} \delta_{ij}^2 \right)^{-1} \right)^{1/2}, \\
\delta_{ij} &= \frac{a_{s_i}}{a_{s_j}} - \frac{a_{s_j}}{a_{s_i}} \,\,\,\mbox{for } i,j \in \{1,\ldots,p\};\, i \neq j.
\end{align*}
\end{thm}
It is worth noting that error bounds of similar nature have been extensively used to study the convergence behavior of various iterative methods; see, e.g.,~\cite{BNPS17,LYS17,Liu2018,zhou2017unified} for some recent developments. Thus, Theorem \ref{lem:EB} can be of independent interest.

To prove Theorem~\ref{lem:EB}, observe that since
\begin{align*}
\dist(\bQ, \bar{\mQ}_{\bq}) &= \dist(\bU_{\bA}^T\bQ\bV_{\bA}, \mQ_{\bq}), \\
\| \bA - \bQ\bA^T\bQ \|_F &= \| \bSigma_{\bA} - (\bU_{\bA}^T\bQ\bV_{\bA})\bSigma_{\bA}^T(\bU_{\bA}^T\bQ\bV_{\bA}) \|_F, \\
\bU_{\bA}^T\bQ\bV_{\bA} &\in {\rm St}(d,K),
\end{align*}
it suffices to establish~\eqref{EB:LP-OC} for the case where $\bA$ has the block structure given in~\eqref{block-A} (in particular, we have $\bU_{\bA} = \bI_d$, $\bV_{\bA}=\bI_K$, and $\bar{\mQ}_{\bq} = \mQ_{\bq}$). In view of the structure of $\mQ_{\bq}$ given in~\eqref{set:Q-p}, a natural idea is to first consider the partition $\bQ = \begin{bmatrix} \bQ_1 & \bQ_2 \\ \bQ_3 & \bQ_4 \end{bmatrix} \in {\rm St}(d,K)$ as in~\eqref{block-Q} and observe that
\be \label{eq:EB-decomp}
\dist^2(\bQ,\mQ_{\bq}) = \dist^2(\bQ_1,\mQ_{\bq}^1) + \|\bQ_2\|_F^2 + \|\bQ_3\|_F^2 + \min_{\bV \in {\rm St}(d-r,K-r)} \| \bQ_4 - \bV \|_F^2,
\ee
where
\be \label{eq:calQ-1}
\mQ_{\bq}^1 = \left\{ \bU\diag(\bq) \bU^T:\ \bU = \bldg(\bU_{1},\dots,\bU_{p}), \,\,\, \bU_i\in \mO^{h_i} \,\,\,\mbox{for}\,\,\, i \in \{1,\ldots,p\} \right\}.
\ee
Then, it suffices to bound each of the terms on the right-hand side of~\eqref{eq:EB-decomp} separately. Let us begin by dispensing with the easy cases.
\begin{prop} \label{prop:blk2-4}
Suppose that $\bA \in \R^{d \times K}$ has the form given in~\eqref{block-A} and $\bQ \in {\rm St}(d,K)$ is partitioned according to~\eqref{block-Q}. Then, the following hold:
\begin{align}
& a_{r}^2\|\bQ_2\|_F^2 \le \|R(\bQ)\|_F^2,\quad a_{r}^2\|\bQ_3\|_F^2 \le \|R(\bQ)\|_F^2, \label{bound:Q2-3} \\
& \min_{\bV \in {\rm St}(d-r,K-r)} \|\bQ_4 - \bV\|_F^2 \le \|\bQ_2\|_F^2. \label{bound:Q4}
\end{align}
\end{prop}
\begin{proof}
We first prove~\eqref{bound:Q2-3}. Using the block structures of $\bA$ in~\eqref{block-A} and $\bQ$ in~\eqref{block-Q} and the fact that $\bQ_1^T\bQ_1 + \bQ_3^T\bQ_3 = \bI_r$, we compute
\begin{align}\label{eq-1:prop-bound-R(Q)}
\|R(\bQ)\|_F^2 & = \|\tbA - \bQ_1\tbA^T\bQ_1\|_F^2 + \|\bQ_1\tbA^T\bQ_2\|_F^2 + \|\bQ_3\tbA^T\bQ_1\|_F^2 + \|\bQ_3\tbA^T\bQ_2\|_F^2 \notag\\
&= \|\tbA - \bQ_1\tbA^T\bQ_1\|_F^2 + \|\bQ_3\tbA^T\bQ_1\|_F^2 + \|\tbA^T\bQ_2\|_F^2.
\end{align} 
This, together with the definition of $\tbA$, implies that 
\[ \|R(\bQ)\|_F^2  \ge \|\tbA^T\bQ_2\|_F^2 \ge a_r^2 \|\bQ_2\|_F^2. \]
Moreover, since 
\be \label{eq:ip-ineq}
0 \le \langle \bB - \bB^T, \bB - \bB^T \rangle = 2 \left( \langle \bB, \bB \rangle - \langle \bB^T, \bB \rangle \right) 
\ee
for any $\bB \in \R^{r \times r}$, we obtain from~\eqref{eq-1:prop-bound-R(Q)} that
\begin{align}
\|R(\bQ)\|_F^2 &\ge \|\tbA - \bQ_1\tbA^T\bQ_1\|_F^2 + \|\bQ_3\tbA^T\bQ_1\|_F^2 \nonumber \\
&= \|\tbA\|_F^2 - 2 \langle \tbA, \bQ_1\tbA^T\bQ_1 \rangle + \|\tbA^T\bQ_1\|_F^2 \nonumber \\
&= \|\tbA\|_F^2 - \langle \bQ_1^T\tbA, \tbA^T\bQ_1 \rangle - \langle \bQ_1^T\tbA, \tbA^T\bQ_1 \rangle + \langle \tbA^T\bQ_1, \tbA^T\bQ_1 \rangle \nonumber \\
&\ge \|\tbA\|_F^2 - \langle \bQ_1^T\tbA, \tbA^T\bQ_1 \rangle. \label{eq-4:prop-bound-R(Q)}
\end{align}
Using the fact that $\bQ_1^T\bQ_1 + \bQ_3^T\bQ_3 = \bI_r$ and invoking~\eqref{eq:ip-ineq},~\eqref{eq-4:prop-bound-R(Q)}, we obtain
\begin{align*}
a_{r}^2\|\bQ_3\|_F^2 &\le \|\tbA\bQ_3^T\|_F^2 = \|\tbA\|_F^2 - \|\tbA\bQ_1^T\|_F^2 \\
&\le \|\tbA\|_F^2 - \langle \bQ_1\tbA^T, \tbA\bQ_1^T \rangle = \|\tbA\|_F^2 - \langle \bQ_1^T\tbA, \tbA^T\bQ_1 \rangle \\
&\le \|R(\bQ)\|_F^2.
\end{align*}
Next, we prove~\eqref{bound:Q4}. Let $\bQ_4=\bU_4\bSigma_4\bV_4^T$ be a thin SVD of $\bQ_4$, where $\bSigma_4 = \diag(\sigma_1,\dots,\sigma_{K-r})$ with $\sigma_1\ge \cdots \ge \sigma_{K-r} \ge 0$ being the singular values of $\bQ_4$, $\bU_4 \in {\rm St}(d-r,K-r)$, and $\bV_4 \in \mO^{K-r}$. Noting that the left-hand side of~\eqref{bound:Q4} is an instance of the orthogonal Procrustes problem \cite{schonemann1966generalized}, we have
\[ \min_{\bV \in {\rm St}(d-r,K-r)} \| \bQ_4 - \bV \|_F^2 = \| \bQ_4 - \bU_4\bV_4^T \|_F^2 = \sum_{i=1}^{K-r} (1-\sigma_i)^2. \]
Using the facts that (i) $(1-x)^2 \le (1-x)^2(1+x)^2$ for any $x\ge0$, (ii) $\bQ_2^T\bQ_2 + \bQ_4^T\bQ_4 = \bI_{K-r}$, and (iii) $\|\bQ_2\| \le \|\bQ\| \le 1$, we obtain
\[ \sum_{i=1}^{K-r} (1-\sigma_i)^2 \le \sum_{i=1}^{K-r} (1-\sigma_i^2)^2 = \|\bI_{K-r} - \bQ_4^T\bQ_4\|_F^2  =  \|\bQ_2^T\bQ_2\|_F^2 \le \|\bQ_2\|_F^2. \]
This completes the proof.
\end{proof}

Now, it remains to bound $\dist^2(\bQ_1,\mQ_{\bq}^1)$. The following technical lemma, whose proof can be found in Appendix~\ref{sec:appen-A}, will be useful for that purpose. Recall that $\delta_{ij} = \tfrac{a_{s_i}}{a_{s_j}}-\tfrac{a_{s_j}}{a_{s_i}}$ for $i,j\in\{1,\ldots,p\}$; $i\not=j$.
\begin{lemma}\label{prop:bound-R(Q)}
Suppose that $\bA \in \R^{d \times K}$ has the form given in~\eqref{block-A} and $\bQ \in {\rm St}(d,K)$ is partitioned according to~\eqref{block-Q}. Then, the following hold:
\begin{align}
a_r^2 \sum_{i=1}^p \|\bQ_{h_ih_i} - \bQ_{h_ih_i}^T \|_F^2 &\le \|R(\bQ)\|_F^2, \label{bound:diag-blk} \\
a_{r}^2 \left( \min_{i,j\in\{1,\ldots,p\} \atop i\not=j} \delta_{ij}^2 \right) \sum_{i=1}^p \sum_{j\neq i} \|\bQ_{h_ih_j}\|_F^2 &\le 6\|R(\bQ)\|_F^2. \label{bound:off-diag-blk}
\end{align}
\end{lemma}
\begin{prop} \label{prop:blk1}
Suppose that $\bA \in \R^{d \times K}$ has the form given in~\eqref{block-A} and let $\bq\in\{\pm1\}^r$ be given. Furthermore, suppose that $\bQ \in {\rm St}(d,K)$ satisfies $\dist(\bQ,\mQ_{\bq}) < 1$ and is partitioned according to~\eqref{block-Q}. Then,
\begin{align}\label{bound:Q1}
\dist^2(\bQ_1,\mQ_{\bq}^1) \le \frac{1}{a_r^2} \left( 10 + 6(6p-5) \left( \min_{i,j\in\{1,\ldots,p\} \atop i\not=j} \delta_{ij}^2 \right)^{-1} \right) \|R(\bQ)\|_F^2.
\end{align}
\end{prop}
\begin{proof}
Based on the block structure of $\bQ_1$ in~\eqref{block-Q} and the definition of $\mQ_{\bq}^1$ in~\eqref{eq:calQ-1}, we have
\begin{align}\label{eq-1:lem-EB}
\dist^2(\bQ_1,\mQ_{\bq}^1) = \sum_{i=1}^p \sum_{j\neq i}\|\bQ_{h_ih_j}\|_F^2 + \sum_{i=1}^p \min_{\bU_i\in \mO^{h_i}} \|\bQ_{h_ih_i}-\bU_i\diag(\bq_i)\bU_i^T\|_F^2.
\end{align}
The first term on the right-hand side of~\eqref{eq-1:lem-EB} can be bounded using~\eqref{bound:off-diag-blk}. Thus, it suffices to bound the second term. Towards that end, let $\tbQ_{h_ih_i} = (\bQ_{h_ih_i}+\bQ_{h_ih_i}^T)/2 \in \S^{h_i}$ for $i = 1,\dots,p$. Since $\bQ_{h_ih_i} - \tbQ_{h_ih_i}$ is skew-symmetric and $\tbQ_{h_ih_i}-\bU_i\diag(\bq_i)\bU_i^T$ is symmetric, we have
\[ \langle \bQ_{h_ih_i} - \tbQ_{h_ih_i}, \tbQ_{h_ih_i}-\bU_i\diag(\bq_i)\bU_i^T \rangle = 0. \]
This, together with~\eqref{bound:diag-blk}, implies that
\begin{align}
& \sum_{i=1}^p \min_{\bU_i\in \mO^{h_i}} \|\bQ_{h_ih_i}-\bU_i\diag(\bq_i)\bU_i^T\|_F^2 \nonumber \\
=& \ \sum_{i=1}^p \|\bQ_{h_ih_i} - \tbQ_{h_ih_i}\|_F^2  +  \sum_{i=1}^p \min_{\bU_i\in \mO^{h_i}} \|\tbQ_{h_ih_i} -\bU_i\diag(\bq_i)\bU_i^T\|_F^2 \nonumber \\
=& \ \frac{1}{4}\sum_{i=1}^p \| \bQ_{h_ih_i}-\bQ_{h_ih_i}^T \|_F^2 + \sum_{i=1}^p \min_{\bU_i\in \mO^{h_i}} \|\tbQ_{h_ih_i} -\bU_i\diag(\bq_i)\bU_i^T\|_F^2  \nonumber \\
\le& \ \frac{1}{4a_r^2} \| R(\bQ) \|_F^2 + \sum_{i=1}^p \min_{\bU_i\in \mO^{h_i}} \|\tbQ_{h_ih_i} -\bU_i\diag(\bq_i)\bU_i^T\|_F^2. \label{eq-2:lem-EB}
\end{align} 

Now, for $i=1,\ldots,p$, let $\tbQ_{h_ih_i} = \bU_{h_i}\Labd_{h_i}\bU_{h_i}^T$ be an eigen-decomposition of $\tbQ_{h_ih_i}\in \S^{h_i}$, where $\Labd_{h_i}=\diag(\lambda^{(h_i)}_{1},\dots,\lambda^{(h_i)}_{h_i})$ with $\lambda^{(h_i)}_{1} \ge \dots\ge \lambda^{(h_i)}_{h_i}$ being the eigenvalues of $\tbQ_{h_ih_i}$. Since $\|\bQ_{h_ih_i}\| \le \|\bQ_1\| \le \|\bQ\| \le 1$ for $i=1,\ldots,p$, we have $|\lambda^{(h_i)}_{k}| \le 1$ for $i = 1,\dots,p$ and $k = 1,\dots,h_i$. We claim that for $i=1,\ldots,p$ and $k=1,\ldots,h_i$, the $k$-th largest eigenvalue of $\tbQ_{h_ih_i}$ has the same sign as the $k$-th largest eigenvalue of $\bU_i\diag(\bq_i)\bU_i^T$. Indeed, if the claim is not true for some $i \in \{1,\ldots,p\}$ and $k\in\{1,\ldots,h_i\}$, then for any $\bU_i\in \mO^{h_i}$, we have
\begin{align*}
\|\tbQ_{h_ih_i} -\bU_i\diag(\bq_i)\bU_i^T\|_F^2 & \ge \|\tbQ_{h_ih_i} -\bU_i\diag(\bq_i)\bU_i^T\|^2 \\
&\ge \left| \lambda^{(h_i)}_{k} - \lambda_k(\bU_i\diag(\bq_i)\bU_i^T) \right| \\
& \ge 1,
\end{align*}
where the second inequality follows from classic perturbation results for eigenvalues of symmetric matrices (see, e.g.,~\cite[Corollary 4.10]{SS90}) and the third follows the fact that the sign of $\lambda_k(\bU_i\diag(\bq_i)\bU_i^T)$ with $\bq_i\in\{\pm 1\}^{h_i}$ is different from that of $\lambda_k^{(h_i)}$. This implies that $\dist(\bQ,\mQ_{\bq}) \ge 1$, which contradicts our assumption that $\dist(\bQ,\mQ_{\bq}) < 1$.

Using the above claim, we can bound
\begin{align}\label{eq-6:lem-EB}
& \min_{\bU_i\in \mO^{h_i}}\|\tbQ_{h_ih_i} - \bU_i\diag(\bq_i)\bU_i^T\|_F^2 \le \sum_{k=1}^{h_i} \left( 1-|\lambda_k^{(h_i)}| \right)^2 \notag \\
\le& \  \sum_{k=1}^{h_i} \left( 1-|\lambda_k^{(h_i)}| \right)^2 \left( 1+|\lambda_k^{(h_i)}| \right)^2  = \| \bI_{h_i} - \tbQ_{h_ih_i}^T\tbQ_{h_ih_i}\|_F^2
\end{align}
for $i=1,\ldots,p$. Let us turn to bound $\sum_{i=1}^p \| \bI_{h_i} - \tbQ_{h_ih_i}^T\tbQ_{h_ih_i}\|_F^2$. Observe that with $\Delta_{h_ih_i}=\bQ_{h_ih_i}-\bQ_{h_ih_i}^T$ for $i=1,\ldots,p$, we have
\begin{align}
& \sum_{i=1}^p\| \bI_{h_i} - \tbQ_{h_ih_i}^T\tbQ_{h_ih_i}\|_F^2 \notag\\
=&  \  \sum_{i=1}^p \left \| \bI_{h_i} - \left( \bQ_{h_ih_i}^T + \frac{1}{2}\Delta_{h_ih_i}  \right) \left( \bQ_{h_ih_i} -  \frac{1}{2}\Delta_{h_ih_i}  \right)    \right\|_F^2 \notag \\
\le& \  3 \sum_{i=1}^p \left(  \|\bI_{h_i} - \bQ_{h_ih_i}^T\bQ_{h_ih_i} \|_F^2 + \| \bQ_{h_ih_i}^T\Delta_{h_ih_i} \|_F^2 + \frac{1}{16} \| \Delta_{h_ih_i}^2 \|_F^2  \right )\notag \\
\le& \ 3  \sum_{i=1}^p \|\bI_{h_i} - \bQ_{h_ih_i}^T\bQ_{h_ih_i} \|_F^2  + \frac{15}{4}\sum_{i=1}^p \| \Delta_{h_ih_i} \|_F^2 \notag \\
\le& \ 3  \sum_{i=1}^p \|\bI_{h_i} - \bQ_{h_ih_i}^T\bQ_{h_ih_i} \|_F^2  + \frac{15}{4a_r^2} \| R(\bQ) \|_F^2, \label{eq-7:lem-EB}
\end{align}
where the second-to-last inequality follows from the fact that $\|\bQ_{h_ih_i}\| \le 1$ and $\|\Delta_{h_ih_i}\| \le 2\|\bQ_{h_ih_i}\|\le 2$, and the last inequality follows from~\eqref{bound:diag-blk}. Continuing, we bound
\begin{align}
& \sum_{j=1}^p \| \bI_{h_j} - \bQ_{h_jh_j}^T\bQ_{h_jh_j}\|_F^2 \notag \\
=& \ \sum_{j=1}^p \left\| \bI_{h_j} - \sum_{i=1}^p \bQ_{h_ih_j}^T\bQ_{h_ih_j} + \sum_{i\neq j} \bQ_{h_ih_j}^T\bQ_{h_ih_j} \right\|_F^2 \notag \\
\le& \ 2 \sum_{j=1}^p \left\| \bI_{h_j} - \sum_{i=1}^p \bQ_{h_ih_j}^T\bQ_{h_ih_j}  \right\|_F^2 + 2  \sum_{j=1}^p \left \| \sum_{i\neq j} \bQ_{h_ih_j}^T\bQ_{h_ih_j} \right\|_F^2 \notag \\
\le& \ 2\|\bI_r - \bQ_1^T\bQ_1\|_F^2 + 2(p-1) \sum_{j=1}^p \sum_{i\neq j}  \| \bQ_{h_ih_j}^T\bQ_{h_ih_j} \|_F^2 \notag \\
\le& \ 2 \|\bQ_3\|_F^2 + 2(p-1) \sum_{i=1}^p \sum_{j \neq i} \| \bQ_{h_ih_j}\|_F^2, \label{eq-8:lem-EB}
\end{align}
where the second inequality follows from the fact that $\{ \bI_{h_j} - \sum_{i=1}^p \bQ_{h_ih_j}^T\bQ_{h_ih_j} \}_{j=1}^p$ are the diagonal blocks of $\bI_r - \bQ_1^T\bQ_1$ and the last is due to $\bQ_1^T\bQ_1 + \bQ_3^T\bQ_3 = \bI_r$, $\|\bQ_3\| \le 1$, and $\| \bQ_{h_ih_j}\| \le 1$ for $i,j\in\{1,\ldots,p\}$.

Upon putting~\eqref{eq-1:lem-EB}--\eqref{eq-8:lem-EB} together and invoking~\eqref{bound:Q2-3} and \eqref{bound:off-diag-blk}, we obtain \eqref{bound:Q1}. This completes the proof.
\end{proof}

We now have all the ingredients to finish the proof of Theorem~\ref{lem:EB}.
\begin{proof}[Proof of Theorem~\ref{lem:EB}]
Let $\bq \in \{\pm1\}^r$ be given. Using~\eqref{eq:EB-decomp} and the results in Propositions~\ref{prop:blk2-4} and \ref{prop:blk1}, we get
\[ 
\dist^2(\bQ,\mQ_{\bq})  \le \frac{1}{a_r^2} \left( 13 + 6(6p-5)  \left( \min_{i,j\in\{1,\ldots,p\} \atop i\not=j} \delta_{ij}^2 \right)^{-1} \right)\|R(\bQ)\|_F^2
\]
for any $\bQ \in {\rm St}(d,K)$ satisfying $\dist(\bQ,\mQ_{\bq}) < 1$. This implies~\eqref{EB:LP-OC}, as desired.
\end{proof}

\subsubsection{From Error Bound to K\L\ Exponent}

Once we have the local error bound~\eqref{EB:LP-OC}, it is rather straightforward to determine the K\L\ exponent at the limiting critical points of Problem \eqref{LP-OC}. We remark that although there are works showing how various error bounds can be used to determine the K\L\ exponent for a host of optimization problems (see, e.g.,~\cite{BNPS17,li2017calculus,Liu2018}), they do not cover our problem setting and hence the results therein cannot be applied directly.

\begin{proof}[Proof of Theorem~\ref{thm:KL-LO-OC}]
Let $\bQ \in {\rm St}(d,K)$ and  $\bQ^* \in \mQ$ be such that $\|\bQ - \bQ^*\|_F < 1$. Furthermore, let $\tbQ^* \in \mQ$ be such that $\dist(\bQ,\mQ)=\|\bQ - \tbQ^*\|_F$. Clearly, we have $\|\bQ - \tbQ^*\|_F < 1$. We claim that $\bQ^*,\tbQ^*\in \bar{\mQ}_{\bq}$ for some $\bq \in \{\pm 1\}^r$. Indeed, if this is not the case, then we have $\bQ^*\in \bar{\mQ}_{\bq}$ and $\tbQ^*\in \bar{\mQ}_{\tbq}$ for some $\bq,\tbq\in\{\pm1\}^r$ with $\bq \not= \tbq$. Since $\dist(\bar{\mQ}_{\bq},\bar{\mQ}_{\tbq}) = \dist(\mQ_{\bq},\mQ_{\tbq})$ (recall the definition of~$\bar{\mQ}_{\bq}$ in Theorem~\ref{lem:EB}), Proposition~\ref{prop:dist-Q} implies that $\|\bQ^* - \tbQ^*\|_F \ge 2$. However, our assumption gives $\| \bQ^*-\tbQ^*\|_F \le  \|\bQ - \bQ^*\|_F + \|\bQ - \tbQ^*\|_F < 2$, which is a contradiction. This establishes the claim.

Next, we claim that $g$ is constant on $\bar{\mQ}_{\bq}$. Indeed, for any $\bW \in \bar{\mQ}_{\bq}$, we have
\[ g(\bW) = \langle \bA, \bW \rangle = \langle \bSigma_{\bA}, \bU_{\bA}^T\bW\bV_{\bA} \rangle = \sum_{i=1}^p \sum_{j=1}^{h_i} a_{s_i}q_{ij}, \]
where the first equality is due to the fact that $\bW \in {\rm St}(d,K)$; the second inequality uses the SVD of $\bA$ in~\eqref{eq:A-svd}; the third inequality follows from the definition of $\bSigma_{\bA}$ (cf.~\eqref{block-A} and \eqref{block-tA}), the fact that $\bU_{\bA}^T\bW\bV_{\bA} \in \mQ_{\bq}$, and the definition of $\mQ_{\bq}$ in~\eqref{set:Q-p}. Upon noting that the rightmost expression does not depend on $\bW$, the claim is established. In particular, we obtain $g(\bQ^*) = g(\tbQ^*)$.

Since $\tbQ^* \in \mQ$, we have $\bA = \tbQ^*\bA^T\tbQ^*$ by Proposition~\ref{prop:subd-G}, which implies that $\bA = \tbQ^*(\tbQ^*)^T\bA$ (see~\eqref{eq-3:lem-crit-LP-OC}). It follows that
\begin{align*}
g(\bQ)-g(\tbQ^*) &= \langle \bA, \bQ-\tbQ^* \rangle = \langle \tbQ^* \bA^T \tbQ^*, \bQ \rangle - \langle \bA, \tbQ^* \rangle \\
&= \langle \bA^T \tbQ^*, (\tbQ^*)^T \bQ - \bI_K \rangle, \\
g(\bQ)-g(\tbQ^*) &= \langle \bA, \bQ-\tbQ^* \rangle = \langle \tbQ^* (\tbQ^*)^T \bA, \bQ \rangle - \langle \bA, \tbQ^* \rangle \\
&= \langle \bA^T \tbQ^*, \bQ^T\tbQ^* - \bI_K \rangle.
\end{align*}
Summing the above two equalities yields
\begin{align*}
|g(\bQ)-g(\tbQ^*)| &= \frac{1}{2} \left| \langle \bA^T \tbQ^*, (\tbQ^*)^T \bQ + \bQ^T\tbQ^* - 2\bI_K \rangle \right| \\
&= \frac{1}{2} \left| \langle \bA^T \tbQ^*, (\bQ - \tbQ^*)^T (\bQ - \tbQ^*) \rangle \right| \\
&\le \frac{1}{2} \|\bA\| \cdot \|\bQ - \tbQ^*\|_F^2 \\
&\le \frac{1}{2} \kappa^2 \|\bA\| \cdot \|R(\bQ)\|_F^2 \\
&\le 2\kappa^2\|\bA\| \cdot \dist^2(\b0, \partial g(\bQ)),
\end{align*}
where the first inequality follows from the Cauchy-Schwarz inequality and the fact that $\|\tbQ^*\| \le 1$; the second inequality follows from the assumption that $\dist(\bQ,\bar{\mQ}_{\bq})=\|\bQ - \tbQ^*\|_F$ and Theorem~\ref{lem:EB}; the last inequality follows from Proposition~\ref{prop:subd-G}. Recalling that $g(\bQ^*) = g(\tbQ^*)$, we establish Theorem~\ref{thm:KL-LO-OC} with $\eta_g = (2\kappa^2\|\bA\|)^{-1/2}$ and $\epsilon_g<1$.
\end{proof}

\subsection{Completing the Proof} \label{subsec:pf-thm-1}

We are now ready to achieve our original goal of characterizing the K\L\ exponent for Problems~\eqref{L1-PCA-1} and \eqref{eq:L1-PCA-2b}.
\begin{proof}[Proof of Theorem~\ref{thm:KL-L1-PCA}]
Recall that the objective function $\ell$ of Problem~\eqref{L1-PCA-1} takes the form $\ell(\bQ) = \min_{i\in\{1,\ldots,2^{nK}\}} \ell_i(\bQ)$, where $\ell_i(\bQ) = \langle \bX\bP_i,\bQ \rangle + \delta_{{\rm St}(d,K)}(\bQ)$. For any $\bQ \in {\rm St}(d,K)$, let 
\[ \mI(\bQ) = \left\{ i \in \{1,\ldots,2^{nK} \} : \ell(\bQ) = \ell_i(\bQ) \right\} \]
denote the set of active indices of $\ell$ at $\bQ$. Furthermore, let $\bQ^* \in {\rm St}(d,K)$ be a limiting critical point of $\ell$. By definition, we have $\min_{i \not\in \mI(\bQ^*)} \ell_i(\bQ^*) > \ell(\bQ^*)$. Thus, there exists an $\epsilon>0$ such that for all $\bQ \in {\rm St}(d,K)$ with $\|\bQ-\bQ^*\|_F \le \epsilon$, we have $\min_{i \not\in \mI(\bQ^*)} \ell_i(\bQ) > \ell(\bQ)$; i.e., $\mI(\bQ) \subseteq \mI(\bQ^*)$. By adapting the proof of~\cite[Theorem 3.1]{li2017calculus} and invoking Theorem~\ref{thm:KL-LO-OC}, we conclude that
\[ \dist(\b0, \partial \ell(\bQ)) \ge \eta_{\ell} | \ell(\bQ) - \ell(\bQ^*) |^{1/2} \]
for all $\bQ \in {\rm St}(d,K)$ with $\| \bQ-\bQ^* \|_F \le \epsilon_{\ell}$, where $\epsilon_{\ell} < \min\{ \epsilon,1 \}$, $\eta_{\ell} = \min_{i \in \mI(\bQ^*)} \eta_{\ell_i}$, and $\eta_{\ell_i}$ is the constant obtained from Theorem~\ref{thm:KL-LO-OC} by taking $g=\ell_i$. This establishes~\ref{exponet-1}.

Next, recall that the objective function $h$ of Problem~\eqref{eq:L1-PCA-2b} takes the form $h(\bP,\bQ) = -\langle \bP, \bX^T \bQ \rangle + \delta_{\mB(n,K)}(\bP) + \delta_{{\rm St}(d,K)}(\bQ)$. Let $\bZ^*=(\bP^*,\bQ^*) \in \mB(n,K)\times{\rm St}(d,K)$ be a limiting critical point of $h$. Furthermore, let $\bZ = (\bP,\bQ) \in \mB(n,K)\times{\rm St}(d,K)$ be such that $\|\bZ - \bZ^*\|_F \le \epsilon_h$ with $\epsilon_h<1$. Since $\| \bP-\bP^* \|_F \le \| \bZ - \bZ^* \|_F < 1$ and $\bP,\bP^* \in \mB(n,K)$, we have $\bP=\bP^*$. Moreover, we have
\be \label{eq:h-subdiff}
\partial h(\bP,\bQ) = \left\{ -\bX^T\bQ + \mN_{\mB(n,K)}(\bP) \right\} \times \left\{ -\bX\bP + \mN_{{\rm St}(d,K)}(\bQ) \right\}
\ee
by~\cite[Proposition 2.1]{attouch2010proximal}. This, together with the fact that $(\b0,\b0) \in \partial h(\bP^*,\bQ^*)$, implies that $\b0 \in -\bX\bP^* + \mN_{{\rm St}(d,K)}(\bQ^*)$; i.e., $\bQ^*$ is a limiting critical point of the function $\bQ \mapsto h(\bP^*,\bQ)$. Hence, by Theorem~\ref{thm:KL-LO-OC}, there exists an $\eta_h>0$ such that 
\begin{align*}
\eta_h |h(\bZ) - h(\bZ^*)|^{1/2} &= \eta_h \left| -\langle \bX\bP^*, \bQ \rangle + \langle \bX\bP^*, \bQ^* \rangle \right|^{1/2} \\
&\le \dist(\b0,\partial h(\bP^*,\bQ)) = \dist(\b0,\partial h(\bZ)).
\end{align*}
This establishes~\eqref{exponet-2}.
\end{proof}

\section{Convergence Analysis of PAMe}\label{sec:pf-thm-2}

Our goal in this section is to prove Theorem~\ref{thm:liner-conv}, which concerns the convergence behavior of our proposed method PAMe (Algorithm \ref{alg:PAMe}). Towards that end, we first combine the characterization of the K\L\ exponent for Problem~\eqref{eq:L1-PCA-2b} in Theorem~\ref{thm:KL-L1-PCA} with the abstract convergence results for descent methods in~\cite{attouch2010proximal,attouch2013convergence} to establish the linear convergence of PAMe to a limiting critical point $(\bP^*,\bQ^*)$ of Problem~\eqref{eq:L1-PCA-2b}. Then, by noting that $(\bP^*,\bQ^*)$ is a solution to certain generalized equation, we obtain a sufficient condition for $\bQ^*$ to be a critical point of Problem~\eqref{L1-PCA-1}.

\subsection{Basic Properties of PAMe}

To study the convergence behavior of PAMe using the analysis framework developed in~\cite{attouch2010proximal,attouch2013convergence}, a key first step is to show that the iterates $\{(\bP^k,\bQ^k)\}_{k\ge0}$ generated by PAMe achieve \emph{sufficient decrease} and satisfy a \emph{relative error} (also referred to as \emph{safeguard} in~\cite{SU15,LYS17,Liu2018}) condition with respect to some potential function. One immediate choice of the potential function is the objective function $h$ of Problem~\eqref{eq:L1-PCA-2b} itself. However, due to the extrapolation step in line~\ref{line:extrap} of Algorithm~\ref{alg:PAMe}, it is not clear whether the sequence $\{h(\bP^k,\bQ^k)\}_{k\ge0}$ satisfies the sufficient decrease and relative error conditions. To circumvent this difficulty, let $\beta\ge0$ be a parameter and consider the potential function $\Psi_{\beta}: \R^{n\times K} \times \R^{d\times K} \times \R^{d\times K} \rightarrow (-\infty,+\infty]$ given by
\begin{align}\label{eq:auxi}
\Psi_{\beta}(\bP,\bQ,\bQ^\prime) = h(\bP,\bQ) + \frac{\beta}{2} \|\bQ - \bQ^\prime\|_F^2.
\end{align}
We note that similar potential functions have previously been used in the convergence analysis of iterative methods with inertial terms/extrapolation steps; see, e.g.,~\cite{pock2016inertial,wen2018proximal,lu2019enhanced,hien2020inertial}. The following result shows that if the step sizes and extrapolation parameters in PAMe are suitably chosen, then there exists a $\beta>0$ such that the sequence $\{\Psi_{\beta}(\bP^k,\bQ^k,\bQ^{k-1})\}_{k\ge0}$ satisfies the two conditions mentioned earlier.
\begin{prop}\label{lem:algo}
Consider the setting of Theorem~\ref{thm:liner-conv}. Let $\bC^k=(\bP^k,\bQ^k,\bQ^{k-1})$ for $k\ge0$. Then, the following hold (recall that $\beta_*$ is given in Theorem~\ref{thm:liner-conv}):
\begin{enumerate}
\item[(a)] The sequence $\{\bC^k\}_{k\ge0}$ is bounded.

\item[(b)] There exists a constant $\kappa_1>0$ such that for $k\ge0$, 
\[ 
\Psi_{\beta_*}(\bC^{k+1}) - \Psi_{\beta_*}(\bC^k) \le  -\kappa_1 \| \bC^{k+1} - \bC^k \|_F^2. 
\] 

\item[(c)] There exists a constant $\kappa_2>0$ such that for $k\ge 0$,
\[ 
\dist(\b0,\partial \Psi_{\beta_*}(\bC^{k+1})) \le \kappa_2 \|\bC^{k+1}-\bC^k\|_F. 
\] 
\end{enumerate}
\end{prop}
\begin{proof}
The proof of (a) is immediate, as $\bC^k \in \mB(n,K) \times {\rm St}(d,K) \times {\rm St}(d,K)$ for $k\ge0$ and both $\mB(n,K)$ and ${\rm St}(d,K)$ are bounded.

To prove (b), we first observe from the updates \eqref{update-P} and \eqref{update-Q} that
\begin{align*}
-\langle \bP^{k+1}, \bX^T\bE^k \rangle + \langle \bP^k, \bX^T\bE^k \rangle &\le -\frac{\alpha_k}{2}\|\bP^{k+1}-\bP^k\|_F^2, \\
-\langle \bP^{k+1}, \bX^T \bQ^{k+1} \rangle + \langle \bP^{k+1}, \bX^T \bQ^k \rangle &\le -\frac{\beta_k}{2}\|\bQ^{k+1}-\bQ^k\|_F^2.
\end{align*}
Let $\bm{\Delta}_{\bP}^{k+1}=\bP^{k+1}-\bP^k$ and $\bm{\Delta}_{\bQ}^{k+1}=\bQ^{k+1}-\bQ^k$ for $k\ge0$. Since $\bE^k = \bQ^k+\gamma_k(\bQ^k-\bQ^{k-1})$, it follows that
\begin{align*}
&\ -\langle \bP^{k+1}, \bX^T \bQ^{k+1} \rangle + \langle \bP^k, \bX^T \bQ^k \rangle \\
\le&\ -\frac{\alpha_k}{2}\|\bP^{k+1}-\bP^k\|_F^2-\frac{\beta_k}{2}\|\bQ^{k+1}-\bQ^k\|_F^2+\gamma_k\langle \bP^{k+1}-\bP^k,\bX^T(\bQ^k-\bQ^{k-1}) \rangle  \\
\le&\ -\frac{\alpha_k}{2}\|\bm{\Delta}_{\bP}^{k+1}\|_F^2-\frac{\beta_k}{2}\|\bm{\Delta}_{\bQ}^{k+1}\|_F^2 + \frac{\gamma_k\alpha_k}{2}\|\bm{\Delta}_{\bP}^{k+1}\|_F^2 + \frac{\gamma_k}{2\alpha_k}\|\bX^T\bm{\Delta}_{\bQ}^{k}\|_F^2 \\ 
\le&\ -\frac{\alpha_k}{2} \left( 1 - \gamma_k \right)\|\bm{\Delta}_{\bP}^{k+1}\|_F^2 -\frac{\beta_k}{2}\|\bm{\Delta}_{\bQ}^{k+1}\|_F^2 + \frac{\gamma_k\|\bX\|^2}{2\alpha_k} \|\bm{\Delta}_{\bQ}^{k}\|_F^2,
\end{align*}
where the second inequality uses the fact that $2\langle \bA,\bB \rangle \le \rho\|\bA\|_F^2 + \|\bB\|_F^2/\rho$ for any $\rho > 0$. Now, by letting $\kappa_1=\min\left\{ \alpha_*( 1 - \gamma^*)/2, \beta_*/4 \right\} > 0$, we obtain
\begin{align*}
&\ \Psi_{\beta_*}(\bC^{k+1}) - \Psi_{\beta_*}(\bC^k) \\
=&\ h(\bP^{k+1},\bQ^{k+1}) + \frac{\beta_*}{2} \|\bQ^{k+1} - \bQ^k \|_F^2  - h(\bP^{k},\bQ^k) - \frac{\beta_*}{2} \|\bQ^{k} - \bQ^{k-1} \|_F^2 \\
\le&\  -\frac{\alpha_k}{2} ( 1 - \gamma_k ) \|\bm{\Delta}_{\bP}^{k+1}\|_F^2 - \frac{1}{2} ( \beta_k - \beta_* ) \|\bm{\Delta}_{\bQ}^{k+1} \|_F^2 - \frac{1}{2}\left(\beta_*  - \frac{\gamma_k\|\bX\|^2}{\alpha_k} \right) \|\bm{\Delta}_{\bQ}^{k}\|_F^2 \\
\le &\ -\kappa_1 \left( \|\bm{\Delta}_{\bP}^{k+1}\|_F^2 + \|\bm{\Delta}_{\bQ}^{k+1}\|_F^2 + \|\bm{\Delta}_{\bQ}^{k}\|_F^2 \right) = -\kappa_1 \|\bC^{k+1} - \bC^k\|_F^2,
\end{align*}
as desired.

Lastly, let us prove (c). Again, using the updates \eqref{update-P} and \eqref{update-Q}, we have
\begin{align}
\b0 &\in -\bX^T\bE^k + \alpha_k(\bP^{k+1}-\bP^k) + \mN_{\mB(n,K)}(\bP^{k+1}), \label{eq-1:lem-algo} \\
\b0 &\in -\bX\bP^{k+1} + \beta_k(\bQ^{k+1}-\bQ^k) + \mN_{{\rm St}(d,K)}(\bQ^{k+1}). \label{eq-2:lem-algo}
\end{align}
By adapting the proof of \cite[Proposition 2.1]{attouch2010proximal}, we obtain
\begin{align*}
\partial \Psi_{\beta_*}(\bP^{k+1},\bQ^{k+1},\bQ^k) &= \{ -\bX^T\bQ^{k+1} + \mN_{\mB(n,K)}(\bP^{k+1}) \} \\
&\quad \times \{ -\bX\bP^{k+1} + \beta_* ( \bQ^{k+1}-\bQ^k ) + \mN_{{\rm St}(d,K)}(\bQ^{k+1}) \} \\
&\quad \times \{ \beta_*(\bQ^k-\bQ^{k+1})\}.
\end{align*}
This, together with \eqref{eq-1:lem-algo} and \eqref{eq-2:lem-algo}, implies that
\begin{align*}
&\  \dist^2(\b0,\partial \Psi_{\beta_*}(\bP^{k+1},\bQ^{k+1},\bQ^k))  \\
\le& \ \|\bX^T(\bE^k-\bQ^{k+1}) - \alpha_k\bm{\Delta}_{\bP}^{k+1}\|_F^2 + \|(\beta_* - \beta_k)\bm{\Delta}_{\bQ}^{k+1}\|_F^2 +  \|\beta_*\bm{\Delta}_{\bQ}^{k+1}\|_F^2 \\
\le &\ 3\alpha_k^2 \|\bm{\Delta}_{\bP}^{k+1}\|_F^2 + ((\beta_k-\beta_*)^2+\beta_*^2+3\|\bX\|^2)\|\bm{\Delta}_{\bQ}^{k+1}\|_F^2 + 3\gamma_k^2 \|\bX\|^2 \|\bm{\Delta}_{\bQ}^{k}\|_F^2.
\end{align*}
By taking $\kappa_2 = \left( \max\{ 3 {\alpha^*}^2, (\beta^*-\beta_*)^2+\beta_*^2+3\|\bX\|^2 \} \right)^{1/2}$, we obtain
\begin{align*}
&\ \dist^2(\b0,\partial \Psi_{\beta_*}(\bP^{k+1},\bQ^{k+1},\bQ^k)) \\
\le& \ \kappa_2^2 \left( \|\bm{\Delta}_{\bP}^{k+1}\|_F^2 + \|\bm{\Delta}_{\bQ}^{k+1}\|_F^2 + \|\bm{\Delta}_{\bQ}^{k}\|_F^2 \right) = \kappa_2^2 \|\bC^{k+1}-\bC^k\|_F^2,
\end{align*}
which implies the desired result.
\end{proof}

\subsection{Linear Convergence of PAMe and Properties of Limit Points}

Proposition \ref{lem:algo} shows that the sequence $\{\bm{C}^k\}_{k\ge0}$ is bounded and satisfies both the sufficient decrease and relative error conditions with respect to the potential function $\Psi_{\beta}$. Thus, a natural next step is to study the convergence behavior of the sequence $\{\bm{C}^k\}_{k\ge0}$ with respect to the potential function $\Psi_{\beta}$ and then use the result to deduce the convergence behavior of the sequence $\{(\bP^k,\bQ^k)\}_{k\ge0}$ with respect to the objective function $h$ of Problem~\eqref{eq:L1-PCA-2b}. To begin, let us prove two technical lemmas. The first establishes a relationship between the limiting critical points of $h$ and $\Psi_{\beta}$.

\begin{lemma}\label{lem:rela-H-Psi}
Let $\beta>0$ be given. Suppose that $(\bP,\bQ,\bQ') \in \mB(n,K) \times {\rm St}(d,K) \times \R^{d\times K}$ is a limiting critical point of $\Psi_{\beta}$. Then, we have $\bQ=\bQ'$. Moreover, $(\bP,\bQ,\bQ)$ is a limiting critical point of $\Psi_{\beta}$ if and only if $(\bP,\bQ)$ is a limiting critical point of $h$.
\end{lemma}
\begin{proof}
Recall that
\begin{align*}
\partial \Psi_{\beta}(\bP,\bQ,\bQ') &= \{ -\bX^T\bQ + \mN_{\mB(n,K)}(\bP) \} \times \{ -\bX\bP + \beta ( \bQ - \bQ' ) + \mN_{{\rm St}(d,K)}(\bQ) \} \\
&\quad \times \{ \beta ( \bQ' - \bQ )\}.
\end{align*}
Thus, if $\beta>0$ and $(\b0,\b0,\b0) \in \partial \Psi_{\beta}(\bP,\bQ,\bQ')$, then we must have $\bQ=\bQ'$. Moreover, we have $(\b0,\b0,\b0) \in \partial \Psi_{\beta}(\bP,\bQ,\bQ)$ if and only if
\[ (\b0,\b0) \in \{ -\bX^T\bQ + \mN_{\mB(n,K)}(\bP) \} \times \{ -\bX\bP  + \mN_{{\rm St}(d,K)}(\bQ) \}. \]
By~\eqref{eq:h-subdiff}, the latter condition holds if and only if $(\b0,\b0) \in \partial h(\bP,\bQ)$.
\end{proof}
The second is motivated by the update of the block variable $\bP$ in Algorithm \ref{alg:PAMe} and shows that a limit point of the sequence $\{\bP^k\}_{k\ge0}$ satisfies certain fixed-point inclusion.
\begin{lemma}\label{lem:rela-PCA}
Let $\alpha>0$ be given. Suppose that the sequences $\{\bP^k\}_{k\ge0}$ and $\{\bY^k\}_{k\ge0}$ satisfy
\begin{align*}
& \bP^k \in \mB(n,K), \quad \bY^k \in \R^{n \times K}, \quad \bP^{k+1} \in \sign(\bP^k + \bY^k/\alpha) \quad \mbox{for } k=0,1,\ldots, \\
& \bP^k \rightarrow \bP^*, \quad \bY^k \rightarrow \bY^*.
\end{align*}
Then, we have $\bP^* \in \sign(\bP^* + \bY^*/\alpha)$. Moreover, for any $\bY \in \R^{n \times K}$ satisfying $\bP^* \in \sign(\bY)$, we have $\bP^* \in \sign(\bP^* + \bY/\alpha)$.
\end{lemma}
\begin{proof}
Let $i \in \{1,\ldots,n\}$ and $j\in\{1,\ldots,K\}$ be arbitrary. If $\bP_{ij}^* + \bY_{ij}^*/\alpha = 0$, then $\sign(\bP_{ij}^* + \bY_{ij}^*/\alpha) = \{-1,1\}$ by definition. Since $\bP^* \in \mB(n,K)$, we have $\bP_{ij}^* \in \sign(\bP_{ij}^* + \bY_{ij}^*/\alpha)$. On the other hand, if $\bP_{ij}^* + \bY_{ij}^*/\alpha \not= 0$, then the assumption that $\bP^k \rightarrow \bP^*$, $\bY^k \rightarrow \bY^*$ implies $\sign(\bP_{ij}^k + \bY_{ij}^k/\alpha) = \sign(\bP_{ij}^* + \bY_{ij}^*/\alpha)$ for all sufficiently large $k\ge0$. As $\bP^{k+1} \in \sign(\bP^k + \bY^k/\alpha)$ for $k\ge0$, we conclude that $\bP_{ij}^* = \sign(\bP_{ij}^* + \bY_{ij}^*/\alpha)$. This establishes the first claim.

Now, let $\bY \in \R^{n \times K}$ be such that $\bP^* \in \sign(\bY)$. If $\bY_{ij}=0$, then $\bP_{ij}^* = \sign(\bP_{ij}^*)$ trivially. On the other hand, if $\bY_{ij} \not= 0$, then $\bP_{ij}^* = \sign(\bY_{ij}/\alpha) = \sign(\bP^*_{ij} + \bY_{ij}/\alpha)$. This establishes the second claim.
\end{proof}
We are now ready to establish the main convergence result for our proposed method PAMe.
\begin{proof}[Proof of Theorem \ref{thm:liner-conv}]
Recall from Theorem~\ref{thm:KL-L1-PCA} that the objective function $h$ of Problem~\eqref{eq:L1-PCA-2b} has a K\L\ exponent of $1/2$ at any of its limiting critical points. Hence, by~\cite[Theorem 3.6]{li2017calculus} and Lemma~\ref{lem:rela-H-Psi}, for any $\beta>0$, the potential function $\Psi_{\beta}$ has a K\L\ exponent of $1/2$ at any of its limiting critical points. It then follows from~\cite[Lemma 2.1]{li2017calculus} that $\Psi_{\beta}$ has a K\L\ exponent of $1/2$ at any $(\bP,\bQ,\bQ') \in \dom(\partial\Psi_{\beta})$. This, together with the results in Proposition~\ref{lem:algo}, allows us to invoke~\cite[Theorem 2.9]{attouch2013convergence} to conclude that under the setting of Theorem~\ref{thm:liner-conv}, the sequence $\{\bC^k\}_{k\ge 0}$ converges to a limiting critical point $(\bP^*,\bQ^*,\bQ^*)$ of the potential function $\Psi_{\beta_*}$. Moreover, by~\cite[Theorem 3.4]{attouch2010proximal}, the rate of convergence is at least linear. It follows from Lemma~\ref{lem:rela-H-Psi} that the sequence $\{ (\bP^k,\bQ^k) \}_{k\ge0}$ converges at least linearly to the limiting critical point $(\bP^*,\bQ^*)$ of Problem~\eqref{eq:L1-PCA-2b}.

Now, suppose that $\alpha_k=\alpha_*$ for $k\ge0$ in Algorithm \ref{alg:PAMe}. According to the update~\eqref{closed-form-P}, the sequence $\{(\bP^k,\bQ^k)\}_{k\ge0}$ satisfies $\bP^{k+1} \in \sign(\bP^k + \bX^T\bE^k/\alpha_*)$, where $\bE^k = \bQ^k + \gamma_k(\bQ^k - \bQ^{k-1})$. Since $(\bP^k,\bQ^k) \rightarrow (\bP^*,\bQ^*)$ and $(\bP^*,\bQ^*)$ is a limiting critical point of $h$, we have $\b0 \in -\bX\sign(\bP^* + \bX^T\bQ^*/\alpha_*) + \mN_{{\rm St}(d,K)}(\bQ^*)$ by~\eqref{eq:h-subdiff} and the result in Lemma~\ref{lem:rela-PCA}; i.e., $\bQ^*$ is a solution to the generalized equation~\eqref{eq:weak-fpt-incl}. In particular, noting that $\bP^* \in \mB(n,K)$, if $\alpha_*$ satisfies~\eqref{eq:alpha-range}, then $\sign(\bP^* + \bX^T\bQ^*/\alpha_*) \subseteq \sign(\bX^T\bQ^*)$. Consequently, we obtain $\b0 \in -\bX\sign(\bX^T\bQ^*) + \mN_{{\rm St}(d,K)}(\bQ^*)$, which, in view of~\eqref{eq:weak-opt}, shows that $\bQ^*$ is a critical point of Problem~\eqref{L1-PCA-1}. Conversely, let $\bar{\bQ}$ be a critical point of Problem~\eqref{L1-PCA-1} that satisfies $\b0 \in -\bX\bP^* + \mN_{{\rm St}(d,K)}(\bar{\bQ})$ and $\bP^* \in \sign(\bX^T\bar{\bQ})$. By Lemma~\ref{lem:rela-PCA}, we have $\bP^* \in \sign(\bP^* + \bX^T\bar{\bQ}/\alpha_*)$. It then follows that $\bar{\bQ}$ is a solution to the generalized equation~\eqref{eq:weak-fpt-incl}.
\end{proof}

\section{Numerical Results} \label{sec:num}

In this section, we report the numerical performance of different L1-PCA algorithms---including our proposed method PAMe, the standard PAM method (see~\eqref{eq:pam-P} and \eqref{update-Q}), the method based on FP iterations (FPM) in~\cite{nie2011robust}, the method pDCAe in~\cite{wen2018proximal}, the inertial proximal alternating linearized minimization (iPALM) method in~\cite{pock2016inertial}, and the Gauss-Seidel-type iPALM (GiPALM) method in~\cite{gao2019gauss}---on both synthetic and real-world datasets. We remark that pDCAe is applied to Problem~\eqref{L1-PCA-1} only in a formal manner, as the objective function $\ell$ is not of the difference-of-convex type. We do not include the inertial proximal block coordinate descent-type algorithm in~\cite{hien2020inertial} in our experiments, as it has essentially the same updates as those of iPALM when applied to Problem~\eqref{eq:L1-PCA-2b}. We also do not include the exact algorithm in~\cite{markopoulos2014optimal} or the algorithm based on BF iterations in~\cite{markopoulos2017efficient} in our experiments, as the datasets we used are too large for them to tackle. All the numerical experiments were conducted on a PC running Windows 10 with an Intel\textsuperscript{\textregistered} Core\texttrademark\, i5-8600 3.10GHz CPU and 16GB memory. Our code runs in MATLAB R2020a and is available at \url{https://github.com/peng8wang/L1-PCA-PAMe}. 

\subsection{Convergence Performance and Solution Quality}\label{sec:num-1}

We begin by studying the convergence performance and solution quality of the different algorithms when applied to both synthetic and real-world instances of the L1-PCA problem. The data matrix $\bX \in \R^{d \times n}$ in a synthetic instance of the L1-PCA problem is generated according to the \emph{fixed effect model} in~\cite{baccini1996l1}. Specifically, for $i=1,\ldots,n$, the $i$-th column of $\bX$ is given by $\bx_i=\bm{z}_i+\bm{e}_i$, where $\bm{z}_i \in \R^d$ is called a \emph{fixed effect} and $\bm{e}_i \in \R^d$ is a \emph{random noise}. The model assumes that the fixed effects $\bz_1,\ldots,\bz_n$ lie on a $K$-dimensional subspace and satisfy $\sum_{i=1}^n\bz_i=\bm{0}$, and that the noise vectors $\bm{e}_1,\ldots,\bm{e}_n$ have entries that are independent and identically distributed (i.i.d.)~according to the Laplace distribution with mean $0$ and variance $\sigma^2$. In our experiments, we generate the fixed effects $\bm{z}_1,\ldots,\bm{z}_n$ in two steps. First, we generate a basis $\bU \in {\rm St}(d,K)$ of the target $K$-dimensional subspace by $\bU = \bm{Y}(\bm{Y}^T\bm{Y})^{-1/2}$, where the entries of $\bm{Y} \in \R^{d\times K}$ are i.i.d.~according to the standard normal distribution. Then, we set $\bz_i=\bm{U}(\bm{a}_i - \bar{\bm a})$ for $i=1,\dots,n$, where each entry of $\bm{a}_i \in \R^K$ is i.i.d.~according to the standard uniform distribution and $\bar{\bm a} = \tfrac{1}{n}\sum_{i=1}^n \bm{a}_i$. With the above setup, we set $\sigma=0.5$, $K=50$ and generate two synthetic instances whose data matrices have dimensions $(n,d)=(4000,2000)$ and $(n,d)=(2000,4000)$, respectively. For the real-world instance, we set $K=20$ and extract a data matrix of dimensions $(n,d)=(15935,62061)$ from the dataset \emph{news20} in LIBSVM \cite{chang2011libsvm}.\footnote{\url{https://www.csie.ntu.edu.tw/~cjlin/libsvmtools/datasets/}}

The parameters of the various algorithms are set as follows. To be fair, we employ the same step sizes when updating the block variables $\bP$ and $\bQ$ in all the PA(L)M-type methods. Specifically, for the two synthetic instances, we set $(\alpha_k,\beta_k) = (10^{-5},10^3)$ and $(\alpha_k,\beta_k) = (10^{-5},10^2)$ for $k\ge0$, respectively; for the real-world instance, we set $(\alpha_k,\beta_k) = (10^{-6},20)$ for $k\ge 0$. The step size for updating the block variable $\bQ$ in pDCAe is set as $\beta_k=1$ for $k\ge 0$. There is no need to choose any step size for FPM. Next, we specify the extrapolation parameters in the PA(L)M-type methods. For PAMe, we set the extrapolation parameter as $1$ for $k \ge 0$. Although such a choice may violate the condition in Theorem~\ref{thm:liner-conv}, it works effectively in our experiments. For iPALM, we set the extrapolation parameters when updating the block variables $\bP$ and $\bQ$ both as $\tfrac{k-1}{k+2}$ for $k\ge1$. Such a choice is motivated by the numerical results in \cite{pock2016inertial}. For GiPALM, we set the extrapolation parameters when updating the block variables $\bP$ and $\bQ$ as $1/2$ and $1/4$ for $k\ge0$, respectively. For pDCAe, we set the extrapolation parameter when updating the block variable $\bQ$ using the fixed restart scheme as suggested in \cite{wen2018proximal} with the fixed restart interval $\bar{T}=10$. In each test, we adopt the same starting point for all the algorithms and terminate them when the Frobenius norm of the difference of two consecutive iterates is less than $10^{-8}$. 

\begin{figure*}[h]
\begin{center}
	\begin{minipage}[b]{0.32\linewidth}
		\centering
		\centerline{\includegraphics[width=\linewidth]{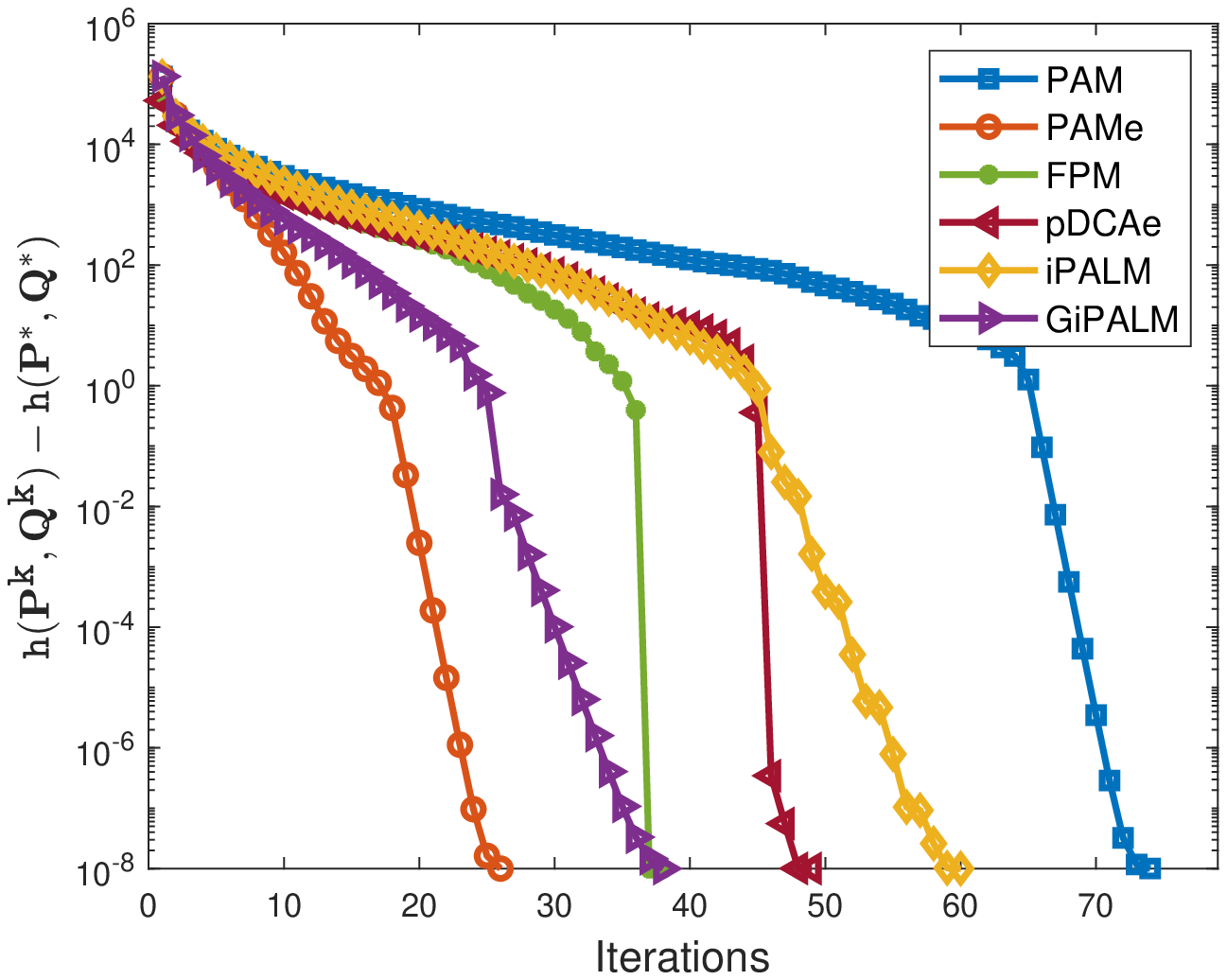}}
		\begin{tabular}{cc}
		\footnotesize (a) & \footnotesize synthetic dataset \\
		\noalign{\vspace{-0.1cm}}
		& \footnotesize $(n,d)=(4000,2000)$
		\end{tabular}
	\end{minipage}
	\begin{minipage}[b]{0.32\linewidth}
		\centering
		\centerline{\includegraphics[width=\linewidth]{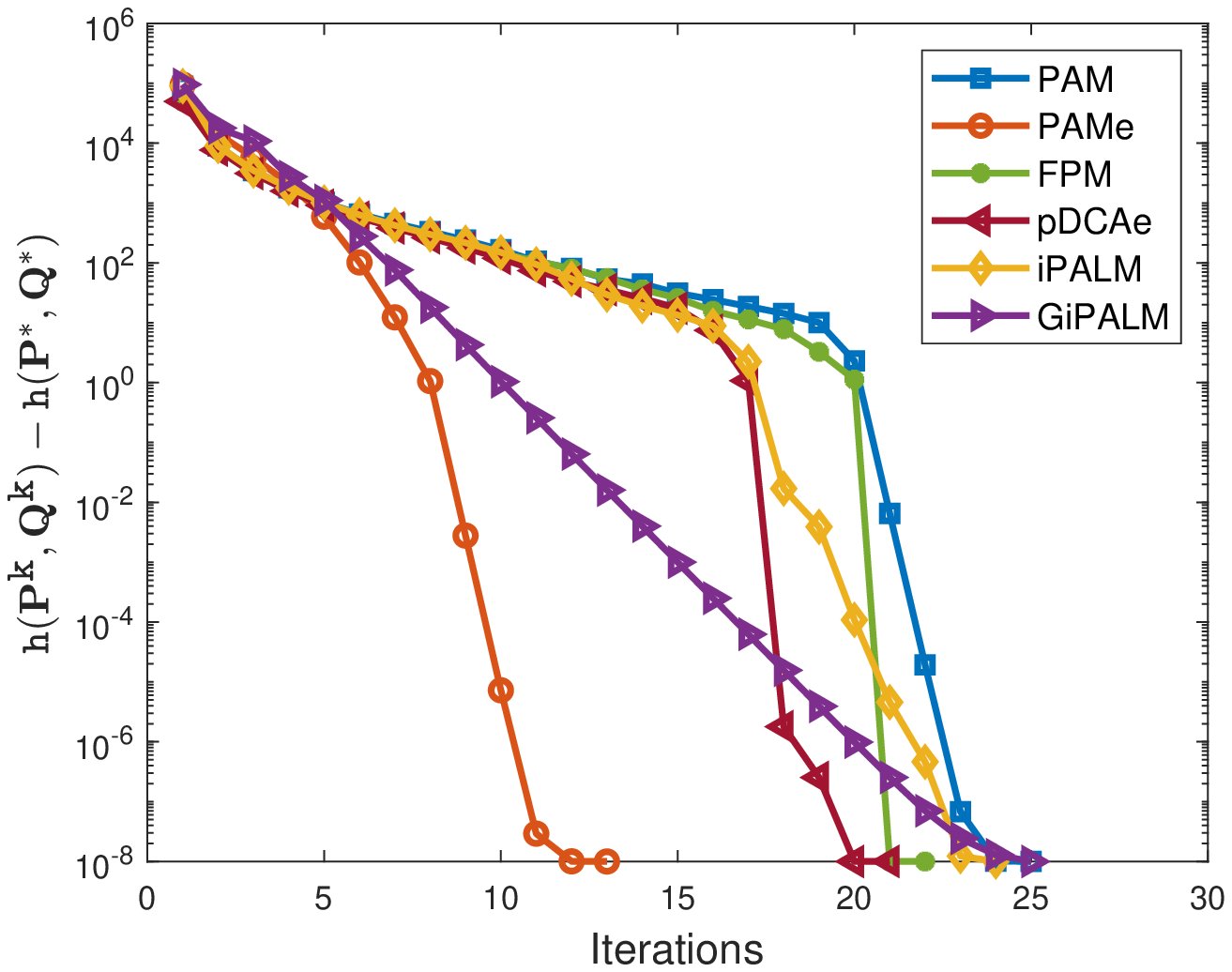}}
		\begin{tabular}{cc}
		\footnotesize (b) & \footnotesize synthetic dataset \\
		\noalign{\vspace{-0.1cm}}
		& \footnotesize $(n,d)=(2000,4000)$
		\end{tabular}
	\end{minipage}
		\begin{minipage}[b]{0.32\linewidth}
		\centering
		\centerline{\includegraphics[width=\linewidth]{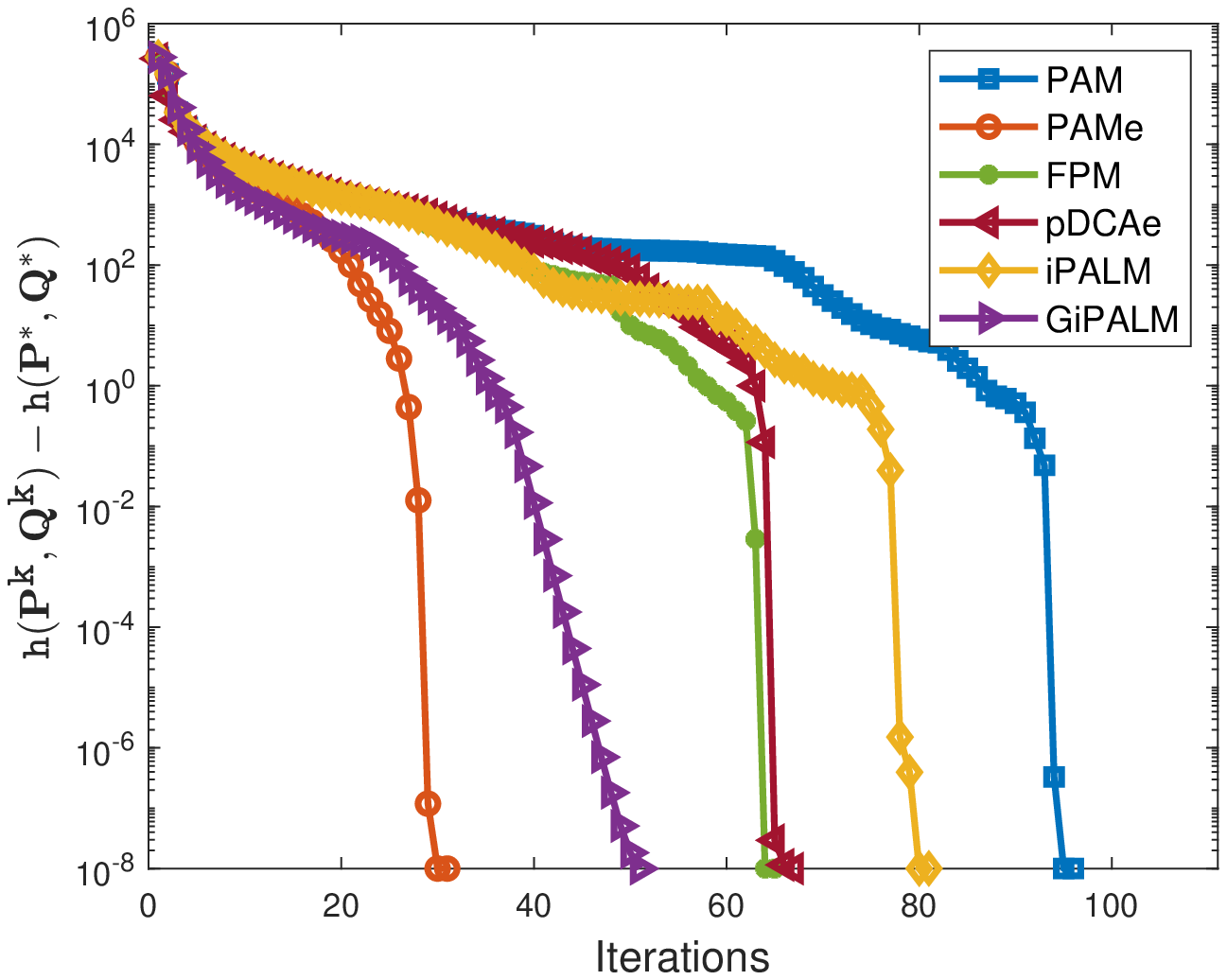}}
		\begin{tabular}{cc}
		\footnotesize (c) & {\footnotesize \emph{news20} dataset} \\
		\noalign{\vspace{-0.1cm}}
		& \footnotesize $(n,d)=(15935,62061)$
		\end{tabular}
	\end{minipage}
\end{center}
	\vskip -0.08in
	\caption{Convergence performance of function values: The $x$-axis is number of iterations; the $y$-axis is function value gap $h(\bP^k,\bQ^k)-h(\bP^*,\bQ^*)$, where $(\bP^*,\bQ^*)$ is the last iterate of the tested method.}
	\label{fig-1}
\end{figure*}

\begin{figure*}[h]
\begin{center}
	\begin{minipage}[b]{0.32\linewidth}
		\centering
		\centerline{\includegraphics[width=\linewidth]{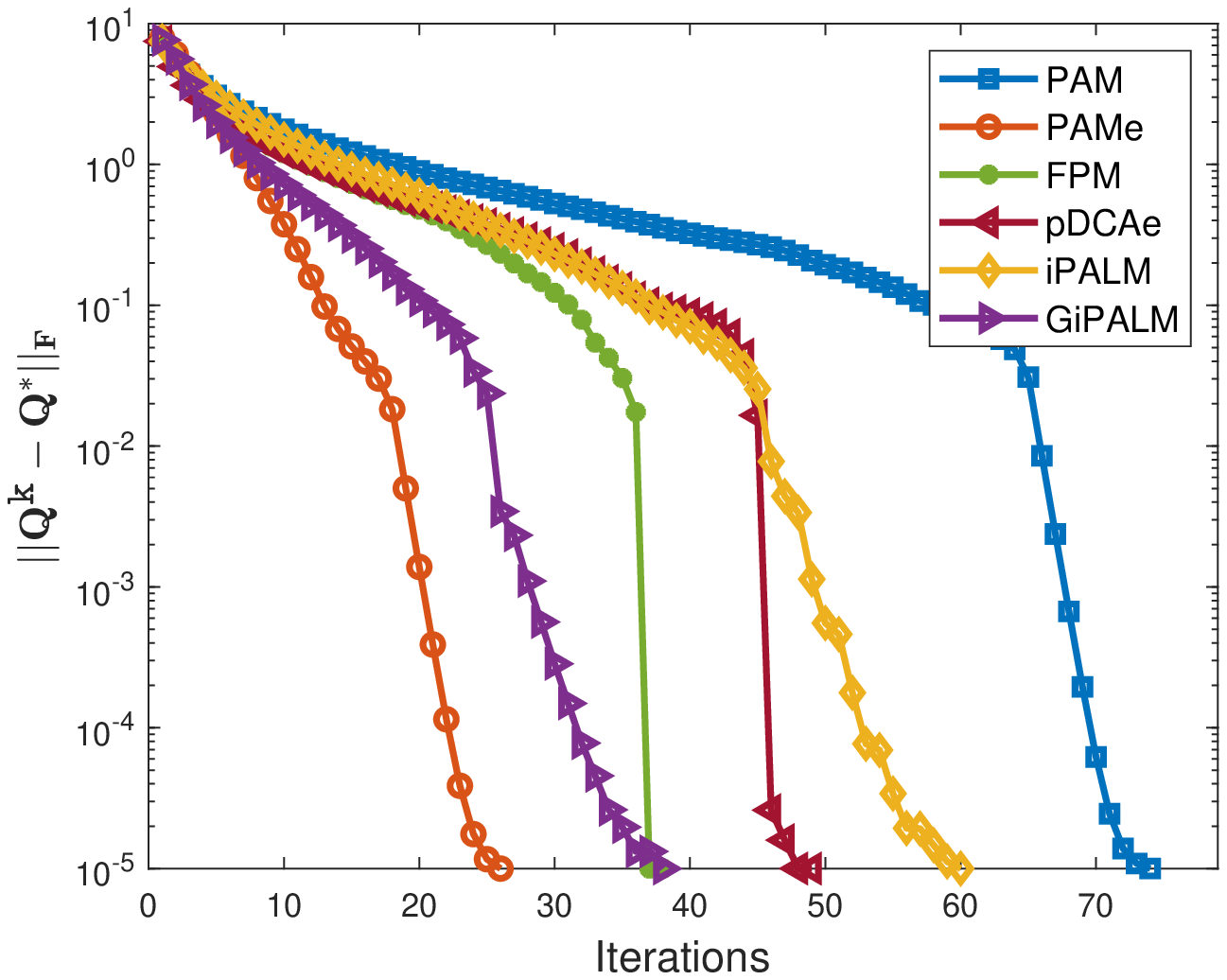}}
		\begin{tabular}{cc}
		\footnotesize (a) & \footnotesize synthetic dataset \\
		\noalign{\vspace{-0.1cm}}
		& \footnotesize $(n,d)=(4000,2000)$
		\end{tabular}
	\end{minipage}
	\begin{minipage}[b]{0.32\linewidth}
		\centering
		\centerline{\includegraphics[width=\linewidth]{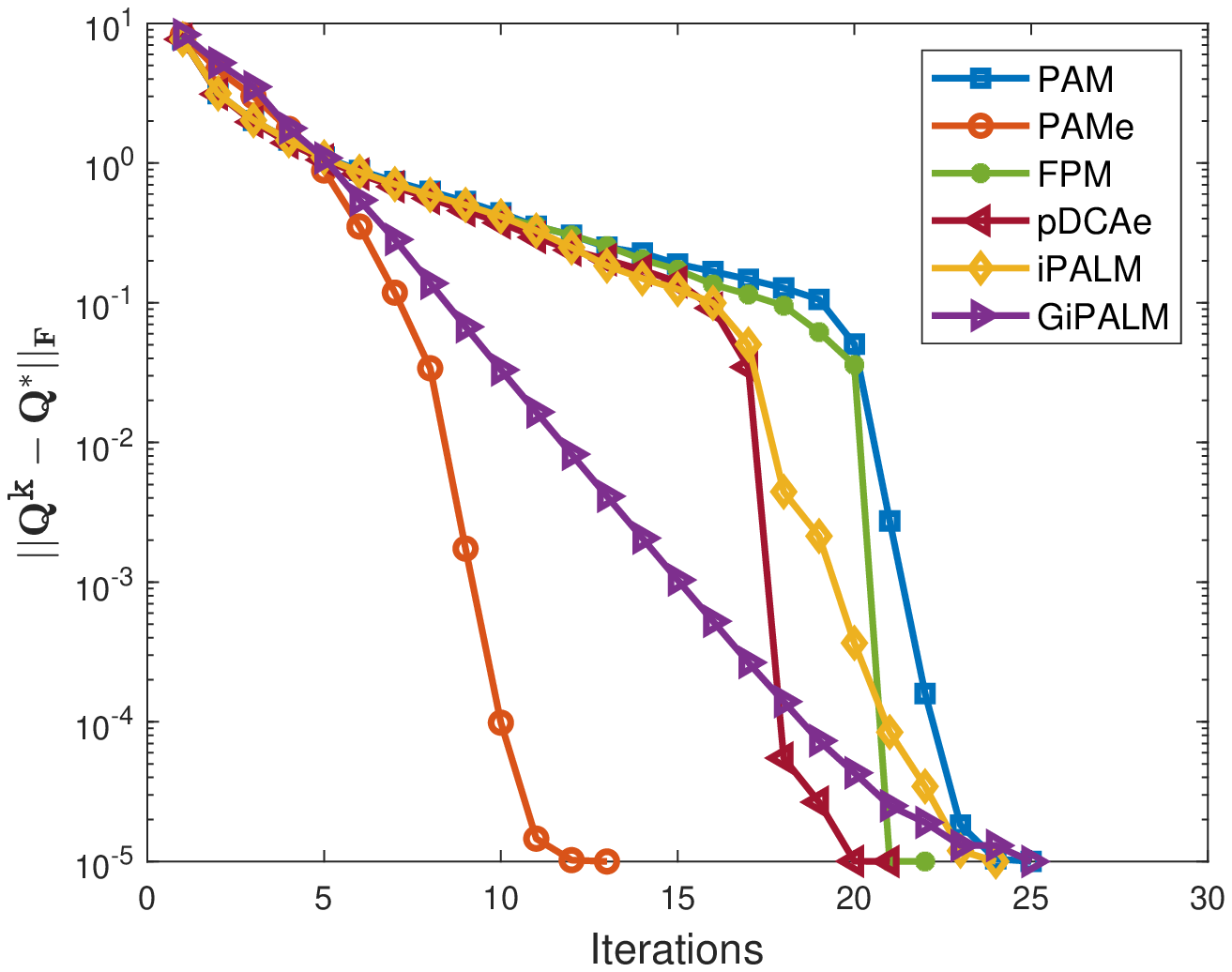}}
		\begin{tabular}{cc}
		\footnotesize (b) & \footnotesize synthetic dataset \\
		\noalign{\vspace{-0.1cm}}
		& \footnotesize $(n,d)=(2000,4000)$
		\end{tabular}
	\end{minipage}
		\begin{minipage}[b]{0.32\linewidth}
		\centering
		\centerline{\includegraphics[width=\linewidth]{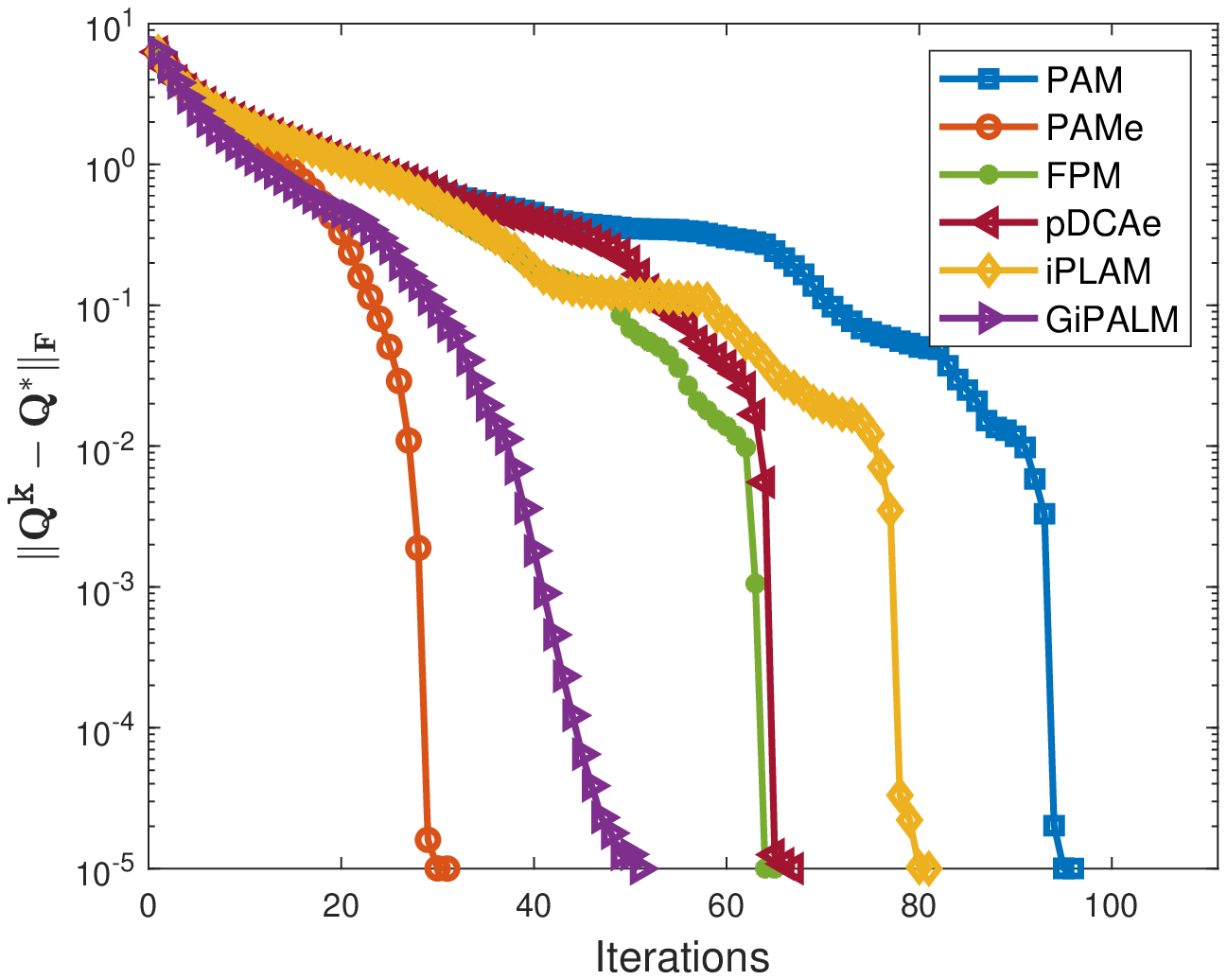}}
		\begin{tabular}{cc}
		\footnotesize (c) & \footnotesize \emph{news20} dataset \\
		\noalign{\vspace{-0.1cm}}
		& \footnotesize $(n,d)=(15935,62061)$
		\end{tabular}
	\end{minipage}
\end{center}
	\vskip -0.08in
	\caption{Convergence performance of iterates: The $x$-axis is number of iterations; the $y$-axis is iterate gap $\|\bQ^k-\bQ^*\|_F$, where $\bQ^*$  is the last iterate of the tested method.}
	\label{fig-2}
\end{figure*}

We plot the function value gap $h(\bP^k,\bQ^k)-h(\bP^*,\bQ^*)$ and the iterate gap $\|\bQ^k-\bQ^*\|_F$ against the iteration number for each tested method in Figures \ref{fig-1} and \eqref{fig-2}, respectively, where $(\bP^*,\bQ^*)$ is the last iterate of the tested method. It can be observed that for all the tested methods, both the sequence of function value gaps and the sequence of iterate gaps converge linearly. In particular, the convergence performance of PAMe supports our linear convergence result in Theorem~\ref{thm:liner-conv}. Moreover, our numerical results demonstrate that PAMe converges substantially faster than the standard PAM method and also faster than FPM, pDCAe, iPALM, and GiPALM.

To compare the quality of the solution returned by each method, we use the \emph{total explained variation} (TEV) measure as in \cite{kim2019simple}, which in our setting is given by
\begin{align*}
{\rm TEV} = \frac{\sum_{i=1}^K \bq_i^T\bX^T\bX \bq_i}{\sum_{i=1}^K \bar{\bq}_i^T\bX^T\bX \bar{\bq}_i}.
\end{align*} 
Here, $\bq_i$ is the $i$-th column of the solution returned by the tested method and $\bar{\bq}_i$ is the $i$-th leading eigenvector of $\bX^T\bX$. Table \ref{table-3} summarizes the TEV of the tested methods. It can be observed that the performance of PAMe is comparable to those of the other methods.  

\begin{table}[!htbp]
\caption{Total explained variation of the tested methods.}
\label{table-3}
\begin{center}
\begin{tabular}{c | c  c  c  c c c c r}
\hline
& \footnotesize PAMe & \footnotesize PAM & \footnotesize FPM & \footnotesize pDCAe & \footnotesize iPALM & \footnotesize GiPALM  \\
\hline
\begin{tabular}{c}
\footnotesize synthetic dataset \\
\noalign{\vspace{-0.1cm}}
\footnotesize $(n,d) = (4000,2000)$
\end{tabular} & \footnotesize {\bf 0.8396} & \footnotesize 0.8214 & \footnotesize 0.8219 & \footnotesize 0.8222 & \footnotesize 0.8227 & \footnotesize 0.8277 \\ 
\begin{tabular}{c}
\footnotesize synthetic dataset \\
\noalign{\vspace{-0.1cm}}
\footnotesize $(n,d) = (2000,4000)$
\end{tabular} & \footnotesize 0.7756 & \footnotesize 0.7210 & \footnotesize 0.7227 & \footnotesize 0.7223 & \footnotesize 0.7212 & \footnotesize {\bf 0.7839} \\
\begin{tabular}{c}
\footnotesize \emph{new20} dataset \\
\noalign{\vspace{-0.1cm}}
\footnotesize $(n,d) = (15935,62061)$
\end{tabular} & \footnotesize {\bf 0.5801} & \footnotesize 0.5741 & \footnotesize 0.5725 & \footnotesize 0.5744 & \footnotesize 0.5720 & \footnotesize 0.5705 \\
\hline
\end{tabular}
\end{center}
\end{table}

\subsection{Application to Clustering on a Subspace} 

As suggested in \cite{ding2006r}, another way of evaluating the performance of an L1-PCA algorithm is to study the clustering accuracy of a dataset on the subspace found by the algorithm. The procedure is as follows. First, we apply the L1-PCA algorithm to the given dataset to compute a subspace. Then, we project the data points onto the subspace and perform $k$-means clustering on the projected points. Finally, we record the fraction of data points that are correctly clustered.
In our experiments, we use the real-world datasets \emph{a9a}, \emph{colon-cancer}, \emph{gisette}, \emph{rcv1.binary}, \emph{real-sim}, and \emph{w8a} in LIBSVM \cite{chang2011libsvm}, whose dimensions can be found in Table \ref{table-1}. In each of these datasets, the data points are given one of two possible labels. These labels serve as the ground truth and naturally divide the data points into two clusters. The dimension $K$ of the subspace used by L1-PCA to capture the variation in the data matrix $\bX \in \R^{d\times n}$ is chosen such that the fraction of total variance explained by the leading $K$ singular values of $\bX$ is not less than $0.8$; i.e., $K$ satisfies $\sum_{k=1}^K\sigma_k^2\ge 0.8 \sum_{k=1}^p \sigma_k^2$, where $p=\min\{n,d\}$ and $\sigma_1 \ge \cdots \ge \sigma_p \ge 0$ are the singular values of $\bX$. If $p$ is so large (say, $p \ge 10000$) that it becomes too expensive to compute all the singular values of $\bX$, we simply set $K=50$.

The step sizes used by the PA(L)M-type methods are listed in Table \ref{table-1}. The step size for updating the block variable $\bQ$ in pDCAe is given by $\beta_k$ in Table \ref{table-1}. We use the same extrapolation parameters for PAMe, pDCAe, iPALM, and GiPALM as those in Subsection \ref{sec:num-1}. 
We terminate the tested methods when either the number of iterations reaches 1000 or the Frobenius norm of the difference of two consecutive iterates is less than $10^{-6}$. To compare the computational efficiency and clustering accuracy of the tested methods, we record their CPU times and ratios of correctly clustered points, averaged over 10 randomly chosen initial points, in Tables \ref{table-2} and \ref{table-4}, respectively. It can be observed that the CPU time consumed by PAMe is generally less than those consumed by the other methods on the tested data sets, especially on \emph{rcv1.binary}, \emph{real-sim}, and \emph{w8a}. Moreover, the clustering accuracy of PAMe is comparable to those of the other methods. These demonstrate the efficiency and efficacy of PAMe when performing clustering on a subspace. 

\begin{table}[!htbp]
\caption{Dimensions of dataset, dimension of subspace $K$, and step size parameters $\alpha_k,\beta_k$.}
\label{table-1}
\begin{center}
\begin{small}
\begin{tabular}{l | l | c | c| c cr}
\hline
& $(n,d)$ & $K$ & $\alpha_k$ & $\beta_k$  \\
\hline
\emph{a9a} & (32561,\ 123) & 6 & $10^{-8}$ & $0.1$  \\ 
\emph{colon-cancer} & (62,\ 2000) & 9 & $10^{-6}$ & 1 \\
\emph{gisette} & (6000,\ 5000) &  1  & $10^{-6}$ & 1 \\
\emph{rcv1.binary} & (20242,\ 47236) &  50 & $10^{-10}$ & 10\\ 
\emph{real-sim} & (72309,\ 20958) &  50  & $10^{-10}$  & 1 \\ 
\emph{w8a} & (49749,\ 300) & 39  & $10^{-10}$ & $1$  \\  
\hline
\end{tabular}
\end{small}
\end{center}
\end{table}

\begin{table}[!htbp]
\caption{CPU time (in seconds) of the tested methods.}
\label{table-2}
\begin{center}
\begin{small}
\begin{tabular}{l | cccc cccc} 
\hline
 & PAMe & PAM &  FPM & pDCAe & iPALM & GiPALM \\                                      
\hline
\emph{a9a} & {\bf 0.12} & 0.19 & 0.22 & 0.17 & 0.21 & 0.14     \\ 
\emph{colon-cancer} & 0.02 & 0.02 & {\bf 0.01} & 0.02 & 0.05 & 0.04 \\
\emph{gisette} & {\bf 0.18} & {\bf 0.18} & 0.19 & 0.19 & {\bf 0.18} &  0.72 \\
\emph{rcv1.binary} & {\bf 7.24} & 24.68 & 27.84 & 25.12 & 29.34 & 16.16 \\
\emph{real-sim} & {\bf 14.95} & 104.4 & 113.2 & 78.02 & 92.93 & 51.25 \\ 
\emph{w8a}  & {\bf 4.43} & 19.38 & 17.42 & 17.10 & 17.77 & 11.88   \\ 
\hline
\end{tabular}
\end{small}
\end{center}
\end{table}

\begin{table}[!htbp]
\caption{Clustering accuracy (i.e., fraction of correctly clustered data points) of the tested methods.}
\label{table-4}
\begin{center}
\begin{small}
\begin{tabular}{l | cccc cccc} 
\hline
 & PAMe & PAM &  FPM & pDCAe & iPALM & GiPALM \\                                      
\hline
\emph{a9a}  & 0.7108 & {\bf 0.7124} & {\bf 0.7124}  & {\bf 0.7124} & {\bf 0.7124} & 0.7104  \\ 
\emph{colon-cancer} & {\bf 0.5532} & 0.5354 & 0.5371 & 0.5419 & 0.5403 & {\bf 0.5532} \\
\emph{gisette} & {\bf 0.5705} & {\bf 0.5705} & {\bf 0.5705} & {\bf 0.5705} & {\bf 0.5705} & {\bf 0.5705} \\
\emph{rcv1.binary} & 0.5885 & 0.5885 & 0.5862 & 0.5884 & {\bf 0.5886} & 0.5883 \\
\emph{real-sim} & {\bf  0.5808} & 0.5807 &  0.5807 & 0.5807 & 0.5807 & 0.5807 \\ 
\emph{w8a}  &  0.7272 & 0.7271 & 0.7271 & {\bf 0.7274} & 0.7273 & 0.7273 \\ 
\hline
\end{tabular}
\end{small}
\end{center}
\end{table}

\section{Concluding Remarks}\label{sec:clud}

In this paper, we proposed a fast iterative method called PAMe to tackle the two-block reformulation~\eqref{L1-PCA-Re} of the L1-PCA problem~\eqref{L1-PCA}. We proved that the sequence of iterates generated by PAMe converges linearly to a limiting critical point of Problem~\eqref{L1-PCA-Re} and gave a sufficient condition under which the said limiting critical point yields a critical point of the original problem~\eqref{L1-PCA}. We also demonstrated the efficiency and efficacy of PAMe via numerical experiments on both synthetic and real-world datasets. As a key step in establishing the linear convergence of PAMe, we showed that the K\L\ exponent at any limiting critical point of Problems~\eqref{L1-PCA} and \eqref{L1-PCA-Re} is $1/2$. This result not only is significant in its own right but also opens the possibility of establishing strong theoretical guarantees on the convergence behavior of other iterative methods (see, e.g.,~\cite{dhanaraj2018novel}) for solving \eqref{L1-PCA} or \eqref{L1-PCA-Re}. Another possible future direction is to consider the design and analysis of fast iterative methods for other $\ell_1$-norm-based variants of PCA (see, e.g.,~\cite{lerman2018overview,tsagkarakis2018l1}).

\bibliographystyle{abbrvnat}

\appendix
\begin{appendices}

\section{Proof of Lemma \ref{prop:bound-R(Q)}} \label{sec:appen-A}
Following the derivation in~\eqref{eq-4:prop-bound-R(Q)} and using the fact that $\|\bQ_1\| \le \|\bQ\| \le 1$, we have 
\begin{align}
\|R(\bQ)\|_F^2 &\ge \|\tbA\|_F^2 - 2 \langle \tbA, \bQ_1\tbA^T\bQ_1 \rangle + \|\tbA^T\bQ_1\|_F^2 \nonumber \\
&\ge 2 \left( \| \tbA^T\bQ_1 \|_F^2 - \langle \bQ_1^T\tbA, \tbA^T\bQ_1 \rangle \right) \nonumber \\
&= \|\tbA^T\bQ_1 - \bQ_1^T\tbA\|_F^2. \label{eq-6:prop-bound-R(Q)}
\end{align}
Now, the block structures of $\tbA^T\bQ_1$ and $\bQ_1^T\tbA$ in~\eqref{eq:sym-diff} and the ordering of the singular values of $\bA$ in~\eqref{rela-a} imply that
\[ 
\|\tbA^T\bQ_1 - \bQ_1^T\tbA\|_F^2 \ge \sum_{i=1}^p \|a_{s_i}(\bQ_{h_ih_i}-\bQ_{h_ih_i}^T)\|_F^2 \ge a_{s_p}^2 \sum_{i=1}^p \left \|\bQ_{h_i h_i} - \bQ_{h_i h_i}^T \right \|_F^2.
\]
This, together with~\eqref{eq-6:prop-bound-R(Q)}, yields~\eqref{bound:diag-blk}.

Using~\eqref{rela-a} and~\eqref{eq:sym-diff} again, we have
\begin{align}\label{eq-8:prop-bound-R(Q)}
\|\tbA^T\bQ_1 - \bQ_1^T\tbA\|_F^2 & \ge \sum_{i=1}^p \sum_{j\neq i} \|a_{s_i}\bQ_{h_ih_j} - a_{s_j}\bQ_{h_jh_i}^T\|_F^2 \notag \\
& \ge a_{s_p}^2 \sum_{i=1}^p \sum_{j\neq i} \left \|\frac{a_{s_i}}{a_{s_j}}\bQ_{h_ih_j} - \bQ_{h_jh_i}^T \right \|_F^2.
\end{align}
In a similar fashion, we get
\begin{align}\label{eq-12:prop-bound-R(Q)}
\|\tbA\bQ_1^T - \bQ_1\tbA^T\|_F^2 \ge a_{s_p}^2 \sum_{i=1}^p \sum_{j\neq i} \left \|\frac{a_{s_j}}{a_{s_i}}\bQ_{h_ih_j} - \bQ_{h_jh_i}^T \right \|_F^2.
\end{align}
Recalling that $\delta_{ij} = \tfrac{a_{s_i}}{a_{s_j}}-\tfrac{a_{s_j}}{a_{s_i}}$ for $i,j\in\{1,\ldots,p\}$; $i\not=j$ and using~\eqref{eq-8:prop-bound-R(Q)} and~\eqref{eq-12:prop-bound-R(Q)}, we bound
\begin{align}
& \left( \min_{i,j \in \{1,\ldots,p\} \atop i\not=j} \delta_{ij}^2 \right) \sum_{i=1}^p \sum_{j\neq i} \|\bQ_{h_ih_j}\|_F^2 \le \sum_{i=1}^p \sum_{j\neq i} \left\|\delta_{ij}\bQ_{h_ih_j} \right\|_F^2 \nonumber \\
\le&\ 2\sum_{i=1}^p \sum_{j\neq i} \left( \left\|\frac{a_{s_i}}{a_{s_j}}\bQ_{h_ih_j} - \bQ_{h_jh_i}^T\right\|_F^2  +  \left\|\frac{a_{s_j}}{a_{s_i}}\bQ_{h_ih_j} - \bQ_{h_jh_i}^T\right\|_F^2 \right) \nonumber \\
\le&\ \frac{2}{a_{s_p}^2}\left( \|\tbA^T\bQ_1 - \bQ_1^T\tbA\|_F^2  + \|\tbA\bQ_1^T - \bQ_1\tbA^T\|_F^2 \right). \label{eq:min-bd}
\end{align}
Similar to the derivation of~\eqref{eq-6:prop-bound-R(Q)}, we have
\begin{align}\label{eq-11:prop-bound-R(Q)}
\|\tbA\bQ_1^T - \bQ_1\tbA^T\|_F^2  \le 2 \left( \|\tbA\|_F^2 - \langle \tbA\bQ_1^T, \bQ_1\tbA^T \rangle \right) \le 2\|R(\bQ)\|_F^2,
\end{align}
where the second inequality follows from \eqref{eq-4:prop-bound-R(Q)} and the fact that $\langle \tbA\bQ_1^T, \bQ_1\tbA^T \rangle = \langle \bQ_1^T\tbA, \tbA^T\bQ_1 \rangle$. Putting~\eqref{eq-6:prop-bound-R(Q)}, \eqref{eq:min-bd}, and~\eqref{eq-11:prop-bound-R(Q)} together, we obtain~\eqref{bound:off-diag-blk}.
\end{appendices}

\end{document}